\documentclass[12pt]{amsart}
\usepackage[utf8]{inputenc}

\font\myfont=cmr12 at 20pt

\title[A Fractal Uncertainty Principle for the STFT and Gabor Multipliers]{\myfont{\textbf{A Fractal Uncertainty Principle for the Short-Time Fourier Transform and Gabor Multipliers}}}
\author{Helge Knutsen}
\email{helge.knutsen@ntnu.no}
\date{April 2022}
\address{Department of Mathematical Sciences, Norwegian University of Science and \mbox{Technology,} \newline 7034 Trondheim, Norway}

\usepackage{amsthm}
\theoremstyle{plain} 
\newtheorem{definition}{Def.}[section]

\newtheorem{theorem}{Theorem}[section]
\newtheorem{corollary}{Corollary}[section]
\newtheorem{lemma}[theorem]{Lemma}

\theoremstyle{definition} 
\newtheorem*{remark}{Remark}
\newtheorem{example}{Example}[section]

\usepackage{amsmath, amssymb} 
\numberwithin{equation}{section} 

\usepackage{mathtools}

\DeclarePairedDelimiter\ceil{\lceil}{\rceil}

\DeclareMathOperator*{\esssup}{ess\,sup} 

\usepackage{enumitem} 

\usepackage{calrsfs} 
\DeclareMathAlphabet{\pazocal}{OMS}{zplm}{m}{n}

\usepackage{bm} 

\usepackage[toc,page]{appendix} 

\usepackage{graphicx}
\graphicspath{ {figures/} }

\usepackage{pst-node}
\usepackage{tikz-cd}

\usepackage{geometry}
\usepackage{afterpage}

\usepackage{verbatim} 

\usepackage{subfigure}

\usepackage[hyphens]{url}

\usepackage{csquotes}

\usepackage[colorlinks,urlcolor=black]{hyperref} 

\begin{document}

\maketitle



\begin{abstract}
We study the fractal uncertainty principle in the joint time-frequency representation, and we prove a version for the Short-Time Fourier transform with Gaussian window on the modulation spaces. This can equivalently be formulated in terms of projection operators on the Bargmann-Fock spaces of entire functions. Specifically for signals in $L^2(\mathbb{R}^d)$, we obtain norm estimates of Daubechies' time-frequency localization operator localizing on porous sets.
The proof is based on the maximal Nyquist density of such sets, and for multidimensional Cantor iterates we derive explicit upper bound asymptotes. Finally, we translate the fractal uncertainty principle to discrete Gaussian Gabor multipliers.
\end{abstract}


\section{Introduction}
The \textit{fractal uncertainty principle} (FUP) was first introduced and developed for the separate time-frequency representation in \cite{Dyatlov_Zahl_2016}, \cite{Bourgain_Dyatlov_2018}, \cite{Dyatlov_Jin_2018}, see also \cite{Jin_Zhang_FUP_2020} for explicit estimates. It states that no signal in $L^2(\mathbb{R})$ can be concentrated near fractal sets in both time and frequency. We reference Dyatlov's detailed introduction to the topic \cite{DyatlovFUP2019}, where fractals sets are defined broadly in terms of either $\delta$-regular sets or almost equivalently in terms of $\nu$-porous sets within the scale bounds $0<h$ to $1$ (see Definition \ref{definition_nu_porosity}). In both definitions, Dyatlov considers families of subsets $T(h),\Omega(h)\subseteq [0,1]$ and formulates the FUP for said families as the lower bound scale $h\to 0$. The FUP is presented as a norm estimate for the localization operator $\chi_{\Omega(h)}\pazocal{F}_{h}\chi_{T(h)}$, where $\chi_{E}$ denotes the characteristic function of a subset $E$, and $\pazocal{F}_{h}$ denotes the dilated Fourier transform $\pazocal{F}_{h}f(\omega) :=\sqrt{h}^{-1}\pazocal{F}f(\omega h^{-1})$. In terms of $\nu$-porosity (see Theorem 2.19 in \cite{DyatlovFUP2019}), for signals in $f\in L^2(\mathbb{R})$ and for families $T(h),\Omega(h)\subseteq[0,1]$ of $\nu$-porous sets on scales $h$ to $1$, there exist constants $C,\beta>0$ only dependent on $\nu>0$ such that
\begin{align*}
    \| \chi_{\Omega(h)}\pazocal{F}_{h}\chi_{T(h)}\|_{L^2(\mathbb{R})\to L^2(\mathbb{R})} \leq C h^{\beta} \ \ \forall \ \ 0<h\leq 1.
\end{align*}
Alternatively, if we disentangle $h$ from the Fourier transform and write $\sqrt{h}$ as $h$, we obtain a statement with regard to families $T(h),\Omega(h)\subseteq[0,h^{-1}]$ of $\nu$-porous sets on scales $h$ to $h^{-1}$, to which there exist constants $C, \beta>0$ only dependent on $\nu>0$ so that 
\begin{align}
    \| \chi_{\Omega(h)}\pazocal{F}\chi_{T(h)}\|_{L^2(\mathbb{R})\to L^2(\mathbb{R})} \leq C h^{\beta} \ \ \forall \ \ 0<h\leq 1.
    \label{FUP_disentangled_statement}
\end{align}
On this form, the FUP more clearly reads as an uncertainty principle as, depending on our choice of $\nu$, the measures of our time and frequency set, $|T(h)|$ and $|\Omega(h)|$, respectively, might tend to infinity as $h\to 0$.

Inspired by the FUP in the separate representation and motivated by the understanding that uncertainty principles should be present regardless of time-frequency representation, we search for analogous results in the joint representation. In particular, we consider and have considered Daubechies' localization operator $P_{\Omega}$ based on the Short-Time Fourier Transform (STFT) with the Gaussian window that projects onto some subset $\Omega$ of the time-frequency plane. In previous installments \cite{HelgeKnutsen_Daubechies_2020}, \cite{HelgeKnutsen_2022}, we have restricted our attention to radially symmetric subsets in $\mathbb{R}^2$, as this yields a known eigenbasis, the Hermite functions, and explicit formulas for the associated eigenvalues. With such insights, we have been able to derive estimates for the operator norm when localizing on radial Cantor iterates that mirror estimate \eqref{FUP_disentangled_statement} but with explicit estimates for the exponent, sometimes even \textit{precise} estimates. The radial assumption has also proved effective for Bergman spaces and by extension for analytic wavelets in \cite{Abreu_Mouayn_Voigtlaender_2022_FUP_Bergman_spaces_analytic_Wavelets}, where direct knowledge of the eigenvalues of the localization operator have produced similar estimates when localizing on the mid-third radial Cantor set. 

In the present paper, however, we abandon the radial assumption and instead consider the more general problem of optimal localization on $\nu$-porous sets in phase space $\mathbb{R}^{2d}$ for arbitrary $d\geq 1$. Although we no longer have knowledge of the eigenvalues of such an operator, the Gaussian window in the STFT in and of itself provides additional structure. Namely, Daubechies' operator $P_{\Omega}:L^2(\mathbb{R}^d)\to L^2(\mathbb{R}^d)$ can equivalently be viewed as a Toeplitz operator on Bargmann-Fock spaces or simply Fock spaces, $\pazocal{F}^2(\mathbb{C}^d)$, of square integrable entire functions with respect to a Gaussian measure. With this perspective, we utilize the subaveraging properties of entire functions to derive estimates of the operator norm $\| P_{\Omega}\|_{\text{op}}$ in terms of the maximal Nyquist density of $\Omega$. These estimates bear resemblance to the estimates in Abreu and Speckbacher' paper \cite{ABREU2021103032_donoho_logan_large_sieve}, which in large part served as inspiration for our approach.
By an inductive scheme, we find that for a family $\Omega(h)\subseteq \mathbb{R}^{2d}$ of $\nu$-porous sets on scales $h$ to $1$, there exists constants $C,\beta>0$ only dependent on $\nu>0$ (and $d$) such that
\begin{align}
    \| P_{\Omega(h)}\|_{\text{op}}\leq C h^{\beta} \ \ \forall \ \ 0<h\leq 1,
    \label{FUP_Daubechies_intro}
\end{align}
which represents a direct analogue of \eqref{FUP_disentangled_statement}, now in the joint representation. In fact, these estimates extend to norm estimates on the Fock space $\pazocal{F}^{p}(\mathbb{C}^{d})$ for generic $p\geq 1$, which in turn yield an FUP not only for $L^2(\mathbb{R}^d)$ but also for the modulation spaces, $M^{p}(\mathbb{R}^{d})$. For more explicit estimates of the exponent in \eqref{FUP_Daubechies_intro}, we specifically consider multidimensional Cantor set constructions, and here the upper bound asymptotes relies on our ability to \textit{directly} compute the Nyquist density of such sets. 

In addition, we present an FUP for Gabor multipliers, which represents a discrete alternative to Daubechies' localization operator based on Gabor frames (see \cite{Feichtinger_Nowak_2003} for an introduction to Gabor multipliers). Approximation properties of such operators been studied in \cite{Grochenig_2009_Gabor_multipliers}, \cite{Cordero_Grochenig_Nicola_2012}, and spectral properties have been studied \cite{Feichtinger_Nowak_Pap_2015}. For our purpose, we consider the closest comparison to the Daubechies' operator with a Gaussian window. Namely, we consider the case when the generating function of the Gabor multiplier also equals Gaussian.  

The paper is organized as follows: Section \ref{section_preliminaries} contains necessary background theory. This includes, a formal introduction to the Fock space and the connection to the STFT (section \ref{section_STFT_Fock_space_preliminaries}), an introduction to Gabor frames and Gabor multipliers (section \ref{section_Gabor_frames_and_multiplisers}) and a precise description of what we mean by "fractal" with the Cantor set constructions as concrete examples (section \ref{section_porous_sets_Cantor_sets}). The results are divided into three sections \ref{section_FUP_joint_representation}, \ref{section_density_of_cantor_sets} and \ref{section_FUP_Gabor_multipliers}. The general FUP for Fock spaces and modulation spaces are formulated in section \ref{section_FUP_joint_representation}. The next section \ref{section_density_of_cantor_sets} is focused on the multidimensional Cantor set constructions, with an FUP formulated specifically for these sets. In the last section \ref{section_FUP_Gabor_multipliers} we show how the FUP can be translated to Gaussian Gabor multipliers.

\section{Preliminaries}
\label{section_preliminaries}
\subsection{From the Short-Time Fourier Transform to the Fock space}
\label{section_STFT_Fock_space_preliminaries}
Consider some fixed window function $\phi: \mathbb{R}^{d}\to \mathbb{C}$, and introduce the basic operations $T_{x}\phi(t) := \phi(t-x)$ and $M_{\omega}\phi(t) := e^{2\pi i \omega\cdot t}\phi(t)$, i.e., time-translation and frequency-modulation, respectively. The \textit{Short-Time Fourier Transform} (STFT) of some signal $f\in L^2(\mathbb{R}^{d})$, with respect to window $\phi$, evaluated at point $(x,\omega)\in \mathbb{R}^{d}\times \mathbb{R}^{d}$, is then given by the inner product \begin{align*}
    V_{\phi}f(x,\omega) :=\langle f, M_{\omega}T_{x}\phi\rangle.
\end{align*}
Observe that if $\phi \equiv 1$, the STFT coincides with the regular Fourier transform. For non-constant windows, however, we obtain a joint time-frequency description of our signal. Furthermore, for $\|\phi\|_2 =1$, the STFT becomes an \textit{isometry} onto some subspace of $L^2(\mathbb{R}^{2d})$, i.e., $\langle V_{\phi}f, V_{\phi}g\rangle_{L^2(\mathbb{R}^{2d})} = \langle f, g\rangle$. In this case, we have an inversion formula, namely
\begin{align}
    f = \int_{\mathbb{R}^{2d}}V_{\phi}f(x,\omega)M_{\omega}T_{x}\phi \ \mathrm{d}x\mathrm{d}\omega,
    \label{STFT_inversion_formula}
\end{align}
where the integral is interpreted in the weak-sense. \textit{Daubechies' time-frequency localization operator}, $P_{S}^{\phi}:L^2(\mathbb{R}^{d})\to L^2(\mathbb{R}^{d})$, with some bounded symbol $S$, is then obtained by modifying the above integrand by the multiplicative weight function $S(x,\omega)$, i.e.,
\begin{align}
    &P_{S}^{\phi}f := \int_{\mathbb{R}^{2d}}S(x,\omega)\cdot V_{\phi}f(x,\omega) M_{\omega}T_{x}\phi \ \mathrm{d}x\mathrm{d}\omega\nonumber \\
    \iff \langle &P_{S}^{\phi}f,g\rangle = \langle S\cdot V_{\phi}f, V_{\phi}g\rangle_{L^2(\mathbb{R}^{2d})} \ \ \forall \ \ g\in L^2(\mathbb{R}^{d}).\label{Daubechies_localization_operator_inner_product_df}
\end{align}
This could equivalently be viewed as modifying the resulting STFT by multiplication by $S$ before inversion. Oftentimes, we consider $S = \chi_{\Omega}$, i.e., the characteristic function of a subset $\Omega$ of the phase space $\mathbb{R}^{2d}$, so that the operator $P_{\Omega}^{\phi}:=P_{\chi_{\Omega}}^{\phi}$ is interpreted as projecting signals onto said time-frequency domain. The associated operator norm $\| P_{\Omega}^{\phi}\|_{\text{op}}$ then measures the optimal localization on $\Omega$.

A popular choice for window function is the \textit{Gaussian} function, which on $\mathbb{R}^{d}$ reads
\begin{align*}
    \phi_0(x) := 2^{d/4}e^{-\pi x^2}, \ \ \text{where } x^2 = x_1^2+\dots+x_d^2 \ \ \text{for } \ x =(x_1,\dots,x_d). 
\end{align*}
With this window choice, we can, in fact, replace Daubechies' operator by a Toeplitz operator on the Bargmann-Fock space or simply Fock space, $\pazocal{F}^{2}(\mathbb{C}^d)$. 

We reference Zhu's book \cite{Zhu_Analysis_on_Fock_spaces} for a detailed introduction to Fock spaces in $\mathbb{C}$.   
For a complex vector $z = x+i\omega = (z_1, \dots, z_{d})\in\mathbb{C}^d$, we distinguish between $z^2 = z\cdot z = z_1^2+\dots+z_{d}^2$ and $|z|^2 = z\cdot \overline{z}= |z_1|^2+\dots+|z_{d}|^2$. Now, for arbitrary $p\geq 1$, let $\mathrm{d}\mu_{p}(z) := e^{-\frac{p}{2}\pi |z|^2} \mathrm{d}A(z)$ denote the Gaussian measure on $\mathbb{C}^{d}$, where $\mathrm{d}A(z)$ is the volume measure $\mathrm{d}x\mathrm{d}\omega$. The associated $L^p$-space, $L^p(\mathbb{C}^d,\mathrm{d}\mu_p)$, is simply denoted by $\pazocal{L}^{p}(\mathbb{C}^d)$. The \textit{Fock space} $\pazocal{F}^{p}(\mathbb{C}^{d})$ is then defined as the Banach space of entire functions $F\in \pazocal{L}^{p}(\mathbb{C}^d)$, with norm\footnote{For consistent and simple notation, we denote the norm in the Fock space by $\|\cdot\|_{\pazocal{L}^p}$ rather than $\|\cdot\|_{\pazocal{F}^p}$. In particular, this is to avoid switching notation for functions in $\pazocal{L}^p(\mathbb{C}^d)\setminus \pazocal{F}^p(\mathbb{C}^d)$, e.g., when we consider $F\cdot \chi_{\Omega}$ for $F\in \pazocal{F}^p(\mathbb{C}^d)\setminus\{0\}$ and $\Omega\subsetneq \mathbb{C}^d$.}
\begin{align*}
    &\|F\|_{\pazocal{L}^{p}} = \left(\int_{\mathbb{C}^{d}}|F(z)|^{p} e^{-\frac{p}{2}\pi |z|^2}\mathrm{d}A(z)\right)^{1/p}<\infty.
\intertext{For $p=\infty$, we let $\pazocal{L}^{\infty}(\mathbb{C}^d)$ denote the space of measurable functions on $
\mathbb{C}^{d}$ such that}
    &\| F \|_{\pazocal{L}^{\infty}} = \esssup \left\{|F(z)| e^{-\frac{\pi}{2}|z|^2} \ \big| \ z\in \mathbb{C}^{d}\right\}< \infty.
\end{align*}
Again, the Fock space $\pazocal{F}^{\infty}(\mathbb{C}^{d})$ is the Banach space of entire functions in $\pazocal{L}^{\infty}(\mathbb{C}^{d})$.
For $p=2$, we find that the Fock space forms a reproducing kernel Hilbert space, with reproducing kernel $K_\xi(z) = e^{\pi z\cdot \overline{\xi}}$ so that $\langle F, K_{\xi} \rangle_{\pazocal{L}^2} = F(\xi)$. Utilizing this kernel, we obtain an orthogonal projection $P: \pazocal{L}^2(\mathbb{C}^{d})\to \pazocal{F}^2(\mathbb{C}^{d})$, defined by
\begin{align*}
    P F(z) := \int_{\mathbb{C}^{d}}K_{\xi}(z) F(\xi) e^{-\pi |\xi|^2}\mathrm{d}A(\xi). 
\end{align*}
For a bounded measurable function $S:\mathbb{C}^{d}(\cong \mathbb{R}^{2d})\to\mathbb{C}$, we define the \textit{Toeplitz operator} $\pazocal{T}_{S}:\pazocal{F}^2(\mathbb{C}^{d})\to\pazocal{F}^2(\mathbb{C}^{d})$, with symbol $S$, by
\begin{align*}
    \pazocal{T}_S F(z) := P(S F)(z) = \int_{\mathbb{C}^{d}}S(\xi)\cdot K_{\xi}(z)F(\xi)e^{-\pi |\xi|^2}\mathrm{d}A(\xi).
\end{align*}
If we initially consider test functions $G\in \pazocal{L}^1(\mathbb{C}^d)\cap \pazocal{F}^2(\mathbb{C}^d)$, then, by Fubini's theorem and the reproducing property of the kernel $K_{\xi}(z)$, the inner product attains the simple form
\begin{align}
    \langle \pazocal{T}_{S}F,G\rangle_{\pazocal{L}^2} = \langle S\cdot F, G\rangle_{\pazocal{L}^2}
    \label{Toeplitz_operator_inner_product}
\end{align}
By a density argument, it follows that \eqref{Toeplitz_operator_inner_product} holds for all $G\in \pazocal{F}^2(\mathbb{C}^d)$, more akin to the inner product \eqref{Daubechies_localization_operator_inner_product_df}.

The precise connection to Daubechies' time-frequency operator with a Gaussian window is established through the Bargmann transform. The \textit{Bargmann transform}, first introduced in \cite{Bargmann_1961}, is an isometric isomorphism $\pazocal{B}:L^2(\mathbb{R}^{d})\to \pazocal{F}^2(\mathbb{C}^{d})$, given by
\begin{align}
    \pazocal{B}f(z) := 2^{d/4}\int_{\mathbb{R}^d}f(t)e^{2\pi t\cdot z-\pi t^2-\frac{\pi}{2}z^2}\mathrm{d}t.
    \label{Bargmann_transform_def}
\end{align}
The transform is related to the STFT with Gaussian window via the following identity (see Proposition 3.4.1 in \cite{GrochenigKarlheinz2001Fota})
\begin{align}
    V_{\phi_0}f(x,-\omega) = e^{\pi i x\cdot \omega}\pazocal{B}f(z)e^{-\frac{\pi}{2}|z|^2} \ \ \text{for } z = x+i\omega.
    \label{Bargmann_transform_STFT_identity}
\end{align}
Define $S^{c}(z) := S(\overline{z})$, so that $\langle P_{S}^{\phi_0} f,g\rangle$ $= \langle \pazocal{T}_{S^c} \pazocal{B}f,\pazocal{B}g\rangle_{\pazocal{L}^2}$, from which the one-to-one correspondence between Daubechies' operator $P_{S}^{\phi_0}$ and the Toeplitz operator $\pazocal{T}_{S^{c}}$ is evident. In particular, we have that their norms coincide. 

In the subsequent discussion, we shall consider Toeplitz operators projecting onto $\Omega\subseteq \mathbb{C}^d$, which we will simply denote by $\pazocal{T}_{\Omega} = \pazocal{T}_{\chi_\Omega}$. By \eqref{Toeplitz_operator_inner_product}, the operator norm is given by
\begin{align}
    \|\pazocal{T}_{\Omega}\|_{\text{op}} = \sup_{\|F\|_{\pazocal{F}^2}=1} \int_{\Omega} |F(z)|^{2} e^{-\pi |z|^2}\mathrm{d}A(z).
\end{align}
For general $p\geq 1$, we consider upper bounds for the quantity 
\begin{align}
   &\frac{\| F\cdot \chi_{\Omega}\|_{\pazocal{L}^p}^p}{\| F\|_{\pazocal{L}^p}^p} =  \frac{\int_{\Omega}|F(z)|^{p}e^{-\frac{p}{2}\pi |z|^2}\mathrm{d}A(z)}{\| F\|_{\pazocal{L}^{p}}^p} \ \ \text{with } \ F\in \pazocal{F}^p(\mathbb{C}^d)\setminus\{0\}.\nonumber
\intertext{By the use of complex interpolation (see Appendix \ref{appendix_complex_interpolation} for details), the above quotients are actually bounded by the estimate for $p=1$, i.e.,}
    &\sup_{F\in \pazocal{F}^p(\mathbb{C}^d)\setminus \{0\}} \frac{\|F\cdot \chi_{\Omega}\|_{\pazocal{L}^p}^p}{\|F\|_{\pazocal{L}^{p}}^p} \leq \sup_{F\in \pazocal{F}^1(\mathbb{C}^d)\setminus \{0\}} \frac{\|F\cdot \chi_{\Omega}\|_{\pazocal{L}^1}}{\|F\|_{\pazocal{L}^{1}}},
    \label{complex_interpolation_inequality_reference}
\end{align}
which allows for estimates \textit{without} any $p$-dependence.
Unsurprisingly, the added structure provided by the Fock space turns out to be beneficial when estimating the norm. Namely, we shall exploit subaveraging properties of \textit{subharmonic} functions.  

\subsection{Gabor frames and Gabor multipliers}
\label{section_Gabor_frames_and_multiplisers}
In general, a family of vectors $\{\phi_\lambda\}_{\lambda\in \Lambda}$ in the Hilbert space $\pazocal{H}$ is called a \textit{frame} if there exist constants $0<A\leq B<\infty$, i.e., frame bounds, such that 
\begin{align*}
    A \| f\|^2 \leq \sum_{\lambda\in\Lambda} |\langle f,\phi_{\lambda}\rangle|^2 \leq B \| f\|^2 \ \ \forall \ \ f\in\pazocal{H}.
\end{align*}
The associated frame operator $\pazocal{S}:\pazocal{H}\to \pazocal{H}$ is given by $\pazocal{S}f := \sum_{\lambda}\langle f,\phi_{\lambda}\rangle \phi_{\lambda}$ with norm $\|\pazocal{S}\|_{\pazocal{H}}\in [A,B]$.
If $A=B$, the frame is called a \textit{tight frame}. By renormalizing the vectors $\phi_{\lambda}\mapsto \phi_{\lambda}/\sqrt{A}$, any tight frame can be turned into a \textit{Parseval frame}, i.e., $A=B=1$, where we also have the representation $f = \sum_{\lambda}\langle f,\phi_{\lambda}\rangle \phi_\lambda$, i.e., $\pazocal{S} = \text{id}$. 

The \textit{Gabor frame} for $L^{2}(\mathbb{R}^{d})$ is based on the idea of discretizing the STFT inversion formula \eqref{STFT_inversion_formula}. 
For this purpose consider a lattice $\Lambda\subseteq \mathbb{R}^{2d}$ of sampled points. Oftentimes, we consider rectangular lattices of the form $\Lambda =  a\mathbb{Z}^d\times b\mathbb{Z}^d$ with parameters $a,b>0$. Further, fix a window function $\phi \in L^2(\mathbb{R}^{d})\setminus \{0\}$ and define the time-frequency shifts $\pi(x,\omega)\phi := M_{\omega}T_{x}\phi$. If the family of time-frequency shifts $\{\pi(\lambda)\phi\}_{\lambda\in\Lambda}$ forms a frame, we call this system a Gabor frame with generating function $\phi$ over lattice $\Lambda$. We may also include a normalization factor based on the density of the lattice. More precisely, to each lattice we can associate a connected neighbourhood of the origin $A_{\Lambda}$ called the fundamental region such that $\cup_{\lambda\in \Lambda}(\lambda+A_{\Lambda})=\mathbb{R}^{2d}$ and $|(\lambda + A_{\Lambda})\cap (\xi + A_{\Lambda})|=0$ whenever $\lambda\neq \xi \in \Lambda$. With the normalization $\sqrt{|A_{\Lambda}|}\{\pi(\lambda)\phi\}_{\lambda\in \Lambda}$, the associated frame operator $\pazocal{S}_{\Lambda}^{\phi}:L^2(\mathbb{R}^d)\to L^2(\mathbb{R}^d)$ reads
\begin{align}
    \pazocal{S}_{\Lambda}^{\phi}f = |A_{\Lambda}|\sum_{\lambda\in \Lambda} \langle f, \pi(\lambda)\phi\rangle \pi(\lambda)\phi.
    \label{Gabor_frame_operator_normalized}
\end{align}
We immediately recognize $\pazocal{S}_{\Lambda}^{\phi}f$ as a Riemann sum of the integral \eqref{STFT_inversion_formula}.
Thus, for a sequence of Gabor frames $\sqrt{|A_{\Lambda_n}}|\{\pi(\lambda)\phi\}_{\lambda\in\Lambda_n}$ where $A_{\Lambda_n}\to 0$ as $n\to \infty$, we expect the discretization \eqref{Gabor_frame_operator_normalized} to converge weakly to the integral \eqref{STFT_inversion_formula} and the frame bounds to tighten. Notably for the sequence of square lattices $\Lambda_{n} = \frac{1}{n}\left(\mathbb{Z}\times \mathbb{Z}\right)$ and $f,\phi$ in Feichtinger's algebra $M^{1}(\mathbb{R}^d)$, Weisz shows in \cite{F_Weisz_2007} that $S^{\phi}_{\Lambda_{n}}f$ converges to $f$ in the $M^{1}$-norm as $n\to \infty$. While $\phi$ remains in $M^1(\mathbb{R}^d)$, this result is extended to $f$ in the modulation space $M^{p,q}(\mathbb{R}^d)$ with convergence in the $M^{p,q}$-norm. In particular, we have convergence in the $L^2$-norm when $f\in L^2(\mathbb{R}^d)$.  

The \textit{Gabor multiplier}, $\mathcal{G}_{\Lambda,b}^{\phi}:L^2(\mathbb{R}^d)\to L^2(\mathbb{R}^d)$, represents a discretization of Daubechies' operator in \eqref{Daubechies_localization_operator_inner_product_df}, where the sum \eqref{Gabor_frame_operator_normalized} is weighted by a bounded symbol $b$ defined on the lattice $\Lambda$, i.e.,
\begin{align*}
    \mathcal{G}_{\Lambda,b}^{\phi}f := |A_{\Lambda}|\sum_{\lambda\in \Lambda} b(\lambda) \langle f,\pi(\lambda)\phi\rangle \pi(\lambda)\phi.
\end{align*}
For localization on a specific subset $\Omega\subseteq \mathbb{R}^{2d}$, we shall consider Gabor symbols $b$ that mimic the behaviour of the characteristic function $\chi_{\Omega}$.
One natural option is to consider the portion of a lattice point region $\lambda+A_{\Lambda}$ containing the subset $\Omega$, i.e.,
\begin{align*}
    b_{\Omega}(\lambda):= \frac{|\Omega \cap (\lambda + A_{\Lambda})|}{|A_{\Lambda}|}\in [0,1].
\end{align*}
Alternatively, we only distinguish between whether the lattice point region $\lambda + A_{\Lambda}$ contains a non-zero part of $\Omega$. That is, we apply the ceiling function $\ceil{\cdot}$, rounding up to the nearest integer, so that
\begin{align*}
    \ceil{b_{\Omega}(\lambda)}\in \{0,1\}.
\end{align*}
Evidently $0\leq b_{\Omega}(\lambda)\leq \ceil{b_{\Omega}(\lambda)}$, from which it is easily verified that the operator norms also satisfy $\|\mathcal{G}_{\Lambda,b_\Omega}^{\phi}\|_{\text{op}} \leq \|\mathcal{G}_{\Lambda,\ceil{b_\Omega}}^{\phi}\|_{\text{op}}$. Thus, when estimating upper bounds for the operator norm in section \ref{section_FUP_Gabor_multipliers}, we only consider the second symbol suggestion, $\ceil{b_{\Omega}}$. Notice that utilizing symbol $\ceil{b_{\Omega}}$ is the same as restricting the summation \eqref{Gabor_frame_operator_normalized} to a subset of the lattice $\Lambda$, namely
\begin{align}
   \Lambda_{\Omega}:= \big\{\lambda \in \Lambda \ \big| \ |\Omega\cap (\lambda+A_{\Lambda})|>0\big\}. 
   \label{lattice_restriction_to_set}
\end{align}
For simplicity, we denote the Gabor multiplier with symbol $\ceil{b_{\Omega}}$ by $\mathcal{G}^{\phi}_{\Lambda, \Omega}$, which, by the above observation, is given by
\begin{align}
    \mathcal{G}_{\Lambda,\Omega}^{\phi}f = |A_{\Lambda}|\sum_{\lambda\in\Lambda_{\Omega}}\langle f, \pi(\lambda)\phi\rangle \pi(\lambda)\phi.
    \label{Gabor_multiplier_subset_formula}
\end{align}

\subsection{Porous sets and Cantor sets}
\label{section_porous_sets_Cantor_sets}
We shall define "fractal sets" in terms of the general notion of $\nu$-porosity. It is based on Definition 2.7 in \cite{DyatlovFUP2019}, adjusted to higher dimensions. Informally, in order for a set to be classified as porous, we require the set to contain gaps or pores within certain scale bounds.   
\begin{definition}
\label{definition_nu_porosity}
\textnormal{($\nu$-porosity)} Suppose $0<\nu<1$, $0\leq \alpha_{\min}\leq \alpha_{\max}\leq \infty$ and $\Omega\subseteq \mathbb{R}^{d}$ is closed. We say that $\Omega$ is $\nu$-porous on scales $\alpha_{\min}$ to $\alpha_{\max}$ if for every ball $B_{r}(x)$ of radius $r\in [\alpha_{\min},\alpha_{\max}]$ there exists a ball $B_{\nu r}(y)\subseteq B_{r}(x)$ of radius $\nu r$ such that $|\Omega\cap B_{\nu r}(y)|=0$.
\end{definition}
Notice that in one dimension, the $\nu$-porous set resembles a Cantor type set, where we are able to remove a $\nu$-portion of any interval inductively down to the lower bound scale. The Cantor sets represent a popular and easy to understand family of fractal sets, which we construct as follows: 
\\\\
Let $M>1$ be a fixed integer, and let $\mathcal{A}$ be a non-empty proper subset of $\{0,1,\dots,M-1\}$. The $n$-iterate ($n$-order) \textit{discrete Cantor set} with base $M$ and alphabet $\mathcal{A}$ is then defined as 
\begin{align*}
    \pazocal{C}_{n}^{(d)}(M,\mathcal{A}):=\left\{ \sum_{j=0}^{n-1}a_j M^j \ \Big| \ a_j\in \mathcal{A}, \ j=0,1,\dots,n-1 \right\} \subseteq \{ 0,1,\dots,M^n-1\}.
\end{align*}
The "continuous" $n$-iterate Cantor set based in the interval $[0,L]$ is given by
\begin{align*}
    \pazocal{C}_n(L,M,\mathcal{A}) := LM^{-n}\cdot\pazocal{C}_{n}^{(d)}(M,\mathcal{A})+[0,LM^{-n}] \ \ \text{for } \ n=0,1,2,\dots 
\end{align*}
The iterates are nested, i.e., $\pazocal{C}_{n+1}\subseteq \pazocal{C}_{n}$, and the (limit) Cantor set is then given by the intersection of all the $n$-iterates.
While the Cantor set itself has measure zero, each $n$-iterate does not. 
If we let $|\mathcal{A}|$ denote the cardinality of the alphabet $\mathcal{A}$, the measure of the $n$-iterate Cantor set is given by
\begin{align*}
    |\pazocal{C}_n(L,M,\mathcal{A})| = \left(\frac{|\mathcal{A}|}{M}\right)^n L.
\end{align*} 
Note that for $M=3$ and $\mathcal{A} = \{ 0,2\}$, we obtain the standard mid-third $n$-iterate Cantor set, with measure $(2/3)^n L$.

Unsurprisingly, the Cantor sets are indeed $\nu$-porous. Below we present a simple estimate for the porosity constant and the scales (see Appendix \ref{appendix_porosity_estimate_cantor_set} for details):

\begin{lemma}
\label{lemma_Cantor_set_porosity}
The $n$-iterate Cantor set $\pazocal{C}_n(L,M,\mathcal{A})$ with base $M>1$ and alphabet size $|\mathcal{A}|<M$, based in the interval $[0,L]$, is $\nu$-porous on scales $M^{-n+1}$ to $\infty$, with any $\nu \leq M^{-2}$. 
\end{lemma}

In multiple dimensions $\mathbb{R}^{2d}$, we consider two possible Cantor set constructions:
\begin{enumerate}
    \item For a ball of radius $R>0$ centered at the origin, we consider the \textit{radially symmetric} $n$-iterate Cantor set as a subset of the form 
\begin{align*}
    \mathcal{C}_n^{2d}(R,M,\mathcal{A}) := \left\{ (x_1,\dots, x_{2d})\in\mathbb{R}^{2d} \ \big| \ \left(x_1^2+\dots+x_{2d}^2\right)^d\in \pazocal{C}_n(R^{2d},M,\mathcal{A})  \right\}.
\end{align*}
In particular, localizing on $\mathcal{C}_n^{1}(R,M,\mathcal{A})$ has been discussed extensively in \cite{HelgeKnutsen_2022}.
In general, these radially symmetric Cantor iterates are constructed such that all annuli that make up the set have the \textit{same} measure. The total measure is given by 
\begin{align*}
    |\mathcal{C}_n^{2d}(R,M,\mathcal{A})|=\left|\pazocal{C}_n\left((\pi R^2)^d/d!,M,\mathcal{A}\right)\right|=\left(\frac{|\mathcal{A}|}{M}\right)^n\frac{(\pi R^2)^d}{d!},
\end{align*}
where we recognize $\frac{(\pi R^2)^d}{d!}$ as the volume of the $2d$-dimensional ball of radius $R$.
    \item Alternatively, we can consider the \textit{Cartesian product} of 1-dimensional $n_j$-iterate Cantor sets, $\{\pazocal{C}_{n_j}(L_j,M_j,\mathcal{A}_j)\}_j$, that is,
\begin{align*}
    \pazocal{C}_{n_1}(L_1,M_1,\mathcal{A}_1)\times \pazocal{C}_{n_2}(L_2,M_2,\mathcal{A}_2)\times \dots \times \pazocal{C}_{n_{2d}}(L_{2d},M_{2d},\mathcal{A}_{2d}).
\end{align*}
If all iterates coincide, the above Cartesian product reduces to $\pazocal{C}_n(L, M,\mathcal{A})^{2d}$ based in the hypercube $[0,L]^{2d}$.

\end{enumerate}

\begin{remark}
In \cite{DyatlovFUP2019} fractal sets are originally defined in terms of $\delta$-regularity for some $0<\delta<1$ (see Definition 2.2 in \cite{DyatlovFUP2019}). While only formulated in $1$-dimension, the notion of $\delta$-regularity can also be extended to higher dimensions. Compared to $\nu$-porosity, this concept offers a different perspective on fractal sets: For instance with regard to the $n$-iterate Cantor set $\pazocal{C}_n(L,M,\mathcal{A})$, we find that the $\delta$ corresponds to fractal dimension (or Hausdorff dimension) of the iterate, namely $\frac{\ln |\mathcal{A}|}{\ln M}$. However, the notion of $\delta$-regularity might appear more abstract than $\nu$-porosity as it does not immediately read as a set containing gaps, and less so the size of those gaps. Nonetheless, as shown in \cite{DyatlovFUP2019} Proposition 2.10, any $\delta$-regular set is $\nu$-porous, and the scales associated to $\delta$-regularity coincide with the $\nu$-porous scales up to multiplicative constants. Thus, formulating the FUP on $\nu$-porous sets directly translates to an FUP on $\delta$-regular sets.   
\end{remark}

\section{Fractal Uncertainty Principle in Joint Representation}
\label{section_FUP_joint_representation}
In this section we present the FUP for the joint time-frequency representation. Initially, in section \ref{section_FUP_Fock_space} we derive the FUP for the Fock spaces. In the subsequent section \ref{section_FUP_modulation_spaces} we introduce the modulation spaces, and translate the FUP for the Fock spaces to an uncertainty principle for the STFT on modulation spaces.
\subsection{Fractal Uncertainty Principle for Fock spaces}
\label{section_FUP_Fock_space}
\begin{theorem}
\label{theorem_FUP_Fock_space}
\textnormal{(FUP for Fock Spaces $\pazocal{F}^{p}(\mathbb{C}^{d})$)}
Let $0<h\leq 1$, and suppose $\Omega(h)\subseteq \mathbb{C}^{d}$ is an $h$-dependent family of sets which is $\nu$-porous on scales $h$ to $1$. Then for all $p\geq 1$ and all $F\in \pazocal{F}^{p}(\mathbb{C}^d)\setminus\{0\}$ there exist constants $C, \beta >0$ only dependent on $\nu$ (and $d$) such that
\begin{align}
    \frac{\|F\cdot \chi_{\Omega(h)}\|_{\pazocal{L}^p}^p}{\|F\|_{\pazocal{L}^p}^p} \leq C h^{\beta} \ \ \forall \ \ 0<h\leq 1.&
\intertext{In particular, for $p=2$, the Toeplitz operator $\pazocal{T}_{\Omega(h)}$ satisfies} 
    \| \pazocal{T}_{\Omega(h)}\|_{\textnormal{op}} \leq C h^{\beta} \ \ \forall \ \ 0<h\leq 1.&
\end{align}
\end{theorem}

The essential property of the Fock space that we utilize is the subaveraging property, where point evaluations can be bounded by an average over the ball. The statement is found in \cite{Zhu_Analysis_on_Fock_spaces} for $\pazocal{F}^{p}(\mathbb{C})$, which generalized to $\pazocal{F}^{p}(\mathbb{C}^{d})$  (see Appendix \ref{appendix_subaveraging_proof} for details) reads:
\begin{lemma}
\label{lemma_subaveraging_in_Fock_space}
For any $F\in \pazocal{F}^{p}(\mathbb{C}^{d})$ and any point $z\in \mathbb{C}^{d}$, we have for all $R>0$ that
\begin{equation}
\begin{aligned}
    |F(z)|^p e^{-\frac{p}{2}\pi |z|^2}\leq \left(\frac{p}{2}\right)^{d}\left(1-e^{-\frac{p}{2}\pi R^2}\sum_{j=0}^{d-1}\frac{\left(\frac{p}{2}\pi R^2\right)^{j}}{j!}\right)^{-1} \int_{B_{R}(z)}|F(\xi)|^p e^{-\frac{p}{2}\pi|\xi|^2}\mathrm{d}A(\xi).
\end{aligned}
\end{equation}
\end{lemma}

Proceeding, we require the following two concepts: For a set $\Omega\subseteq \mathbb{C}^{d}$ and $R>0$, we define the \textit{maximal Nyquist density}, $\rho(\Omega,R)$, by
\begin{align}
    \rho(\Omega,R)&:= \sup_{z\in\mathbb{C}^{d}}|\Omega\cap B_{R}(z)|\leq \max\{|\Omega|,|B_{R}(0)|\}.
    \label{Nyquist_density_definition}\\
\intertext{Further, we define the $R$-\textit{thickened set}, $\Omega_R$, by}
    \Omega_R &:= \Omega + B_{R}(0) = \cup_{z\in \Omega}B_{R}(z).\nonumber
\end{align}
With these notions and Lemma \ref{lemma_subaveraging_in_Fock_space}, we present an upper bound estimate for the integral over $\Omega$.
\begin{lemma}
\label{lemma_general_upper_bound_nu_porocity_thickening}
Suppose $\Omega\subseteq \mathbb{C}^{d}$ measurable. Then for any $F\in \pazocal{F}^{p}(\mathbb{C}^{d})$ and any $R>0$, we have that
\begin{equation}
\begin{aligned}
    \int_{\Omega}|F(z)|^{p}e^{-\frac{p}{2}\pi|z|^{2}}\mathrm{d}A(z) \leq \left(\frac{p}{2}\right)^{d}\frac{\rho(\Omega,R)}{1-e^{-\frac{p}{2}\pi R^2}\sum_{j=0}^{d-1}\frac{\left(\frac{p}{2}\pi R^2\right)^j}{j!}} \int_{\Omega_R}|F(z)|^{p}e^{-\frac{p}{2}\pi|z|^{2}}\mathrm{d}A(z).
\end{aligned}
\end{equation}
\begin{proof}
By Lemma \ref{lemma_subaveraging_in_Fock_space},
\begin{align*}
    \int_{\Omega}|F(z)|^{p}e^{-\frac{p}{2}\pi|z|^{2}}\mathrm{d}A(z) \leq  \left(\frac{p}{2}\right)^{d}&\left(1-e^{-\frac{p}{2}\pi R^2}\sum_{j=0}^{d-1}\frac{\left(\frac{p}{2}\pi R^2\right)^j}{j!}\right)^{-1}\\
    &\cdot\int_{\Omega}\int_{B_{R}(z)}|F(\xi)|^{p}e^{-\frac{p}{2}\pi |\xi|^2}\mathrm{d}A(\xi)\mathrm{d}A(z).
\end{align*}
Since $z\in \Omega$, we have that
\begin{align*}
    \int_{B_{R}(z)}|F(\xi)|^{p}e^{-\frac{p}{2}\pi |\xi|^2}\mathrm{d}A(\xi) = \int_{\Omega_{R}}\chi_{B_{R}(z)}(\xi)\cdot|F(\xi)|^{p}e^{-\frac{p}{2}\pi |\xi|^2}\mathrm{d}A(\xi). \ \ \ \ \ \ 
\end{align*}
By Fubini's theorem, we obtain 

\begin{align*}
    &\int_{\Omega}\int_{B_{R}(z)}|F(\xi)|^{p}e^{-\frac{p}{2}\pi |\xi|^2}\mathrm{d}A(\xi)\mathrm{d}A(z)=\int_{\Omega_{R}}\left(\int_{\Omega}\chi_{B_{R}(\xi)}(z)\mathrm{d}A(z)\right)|F(\xi)|^{p}e^{-\pi \frac{p}{2} |\xi|^2}\mathrm{d}A(\xi)\\
    &=\int_{\Omega_{R}}|\Omega\cap B_{R}(\xi)|\cdot |F(\xi)|^{p}e^{-\frac{p}{2}\pi |\xi|^2}\mathrm{d}A(\xi) \leq \rho(\Omega,R)\int_{\Omega_{R}}|F(\xi)|^{p}e^{-\frac{p}{2}\pi |\xi|^2}\mathrm{d}A(\xi).
\end{align*}

\end{proof}
\end{lemma}

In the above lemma notice that the integral over $\Omega_{R}$ is \textit{always} bounded by the integral over $\mathbb{C}^d$, that is, 
\begin{align*}
    \| F\cdot \chi_{\Omega}\|_{\pazocal{L}^p}^p \leq \left(\frac{p}{2}\right)^{d}&\frac{\rho(\Omega,R)}{1-e^{-\frac{p}{2}\pi R^2}\sum_{j=0}^{d-1}\frac{\left(\frac{p}{2}\pi R^2\right)^j}{j!}} \ \| F\|_{\pazocal{L}^p}^p \ \ \forall \ \ R>0.
\end{align*}
Further, by the interpolation result \eqref{complex_interpolation_inequality_reference}, we can optimize the right-hand side with the estimate for $p=1$ so that
\begin{align}
    \| F\cdot \chi_{\Omega}\|_{\pazocal{L}^p}^p \leq 2^{-d}&\frac{\rho(\Omega,R)}{1-e^{-\frac{1}{2}\pi R^2}\sum_{j=0}^{d-1}\frac{\left(\frac{1}{2}\pi R^2\right)^j}{j!}} \ \| F\|_{\pazocal{L}^p}^p \ \ \forall \ \ R>0.
    \label{Fock_space_nyquist_density}
\end{align}
Regardless, this shows that estimates of the quotient $\| F\cdot \chi_{\Omega}\|_{\pazocal{L}^p}^{p}/\| F\|_{\pazocal{L}^p}^{p}$ can be solely based on estimates of the maximal Nyquist density. E.g., as shown in the subsequent section \ref{section_density_of_cantor_sets}, under certain growth conditions, we are able to obtain good estimates for Nyquist density for the standard Cartesian product of Cantor sets and radial Cantor sets. For the general case of porous sets, however, we also take into account the integral over the thickened set.

Besides the trivial upper bounds of \eqref{Nyquist_density_definition}, if we suppose $\Omega\subseteq \mathbb{C}^{d}$ is $\nu$-porous on scales $\alpha_{\min}$ to $\alpha_{\max}$ and consider a radius $R$ within the scale bounds, we obtain, by Def. \ref{definition_nu_porosity}, the simple estimate
\begin{align*}
    \rho(\Omega,R) \leq \sup_{|z-\xi|\leq (1-\nu )R}\big|B_{R}(z)\setminus B_{\nu R}(\xi)\big|= \left(1-\nu^{2d}\right)|B_{R}(0)| = \left(1-\nu^{2d}\right)\frac{\left(\pi R^2\right)^d}{d!}.
\end{align*}
We apply this upper bound to Lemma \ref{lemma_general_upper_bound_nu_porocity_thickening}, which yields the corollary:

\begin{corollary}
\label{corollary_subaveraging_in_Fock_space_nu_porocity}
Suppose $\Omega\subseteq \mathbb{C}^{d}$ is $\nu$-porous on scales $\alpha_{\min}$ to $\alpha_{\max}$. Then for any radius $R\in [\alpha_{\min},\alpha_{\max}]$ and any function $F\in \pazocal{F}^{p}(\mathbb{C}^{d})$, we have that 
\begin{align}
    \int_{\Omega}|F(z)|^{p}e^{-\frac{p}{2}\pi|z|^{2}}\mathrm{d}A(z) \leq \kappa_d\left(\frac{p}{2}\pi R^2 \right)\cdot \left(1-\nu^{2d}\right) \int_{\Omega_R}|F(z)|^{p}e^{-\frac{p}{2}\pi|z|^{2}}\mathrm{d}A(z),
\end{align}
where the function $\kappa_{d}$ is given by
\begin{align}
    \kappa_{d}(x):=\frac{x^{d}}{d!}\left(1-e^{-x}\sum_{j=0}^{d-1}\frac{x^j}{j!}\right)^{-1}.
    \label{thickening_constant_kappa}
\end{align}
\end{corollary}

The crucial observation, moving forward, is that for some choices of $R>0$, the thickened set $\Omega_{R}$ is itself a porous set provided the original set $\Omega$ is porous. A special case of this observation is presented in Proposition 2.11 \cite{DyatlovFUP2019}. 

\begin{lemma}
\label{lemma_thickening_porous_set}
\textnormal{(Thickening of porous set)} Suppose $\Omega\subseteq \mathbb{C}^{d}$ is $\nu$-porous on scales $\alpha_{\min}$ to $\alpha_{\max}$. For any $0\leq r< \nu\cdot\alpha_{\max}$, consider the $r$-thickened set $\Omega_{r}:=\Omega+B_{r}(0) = \cup_{z\in\Omega}B_{r}(z)$. Then for any $R\in[\alpha_{\min},\alpha_{\max}]$ with $\frac{r}{\nu}<R$, the set $\Omega_{r}$ is $\left(\nu-\frac{r}{R}\right)$-porous on scales $R$ to $\alpha_{\max}$.
\begin{proof}
Consider $\alpha_{\min}\leq R\leq L\leq\alpha_{\max}$. By $\nu$-porosity of $\Omega$, for any $B_{L}(z)$ there exists $B_{\nu L}(\xi)\subseteq B_{L}(z)$ such that $|B_{\nu L}(\xi)\cap\Omega| =0$. After $r$-thickening, we maintain a zero intersection with $\Omega_{r}$ if the radius of $B_{\nu L}(\xi)$ is reduced to $\nu L-r$. Since $L$ is arbitrary and
\begin{align*}
    \left(\nu-\frac{r}{R}\right)L \leq \nu L -r \implies |B_{\left(\nu-\frac{r}{R}\right)L}(\xi)\cap \Omega_{r}| \leq |B_{\nu L-r}(\xi)\cap \Omega_{r}|=0,
\end{align*}
the set $\Omega_{r}$ satisfies the claimed porosity properties. The conditions on $r$ and $R$ ensure that the new porosity is positive and that the scale bounds are valid. 
\end{proof}
\end{lemma}

We are now ready to prove Theorem \ref{theorem_FUP_Fock_space}.

\begin{proof}
(Theorem \ref{theorem_FUP_Fock_space}) 
By the interpolation result \eqref{complex_interpolation_inequality_reference}, it suffices to prove the statement for the case $p=1$. Nonetheless, for notational convenience, we will rather consider the case $p=2$ as the only difference in the proof between $p=2$ and general $p\geq 1$ is that $\kappa_{d}(x)$ is replaced by $\kappa_{d}(\frac{p}{2}x)$.

Recall that for some $0<h\leq 1$ we presume $\Omega = \Omega(h)$ is $\nu$-porous on scales $h$ to $1$. For some finite sequence of radii $\{R_j\}_{j=1}^{n}$, we consider an associated sequence of thickened sets $\{\Omega^{(j)}\}_{j=0}^{n-1}$, given by
\begin{align*}
    \Omega^{(0)} :=\Omega \ \ \text{and } \ \Omega^{(j)}:=\Omega_{R_1+R_2+\dots+R_j} \ \ \text{for } j=1,2,\dots,n-1.
\end{align*}
We chose the radii such that the set $\Omega^{(j)}$ maintains the porous property, in the sense that 
\begin{align*}
    \Omega^{(j)} \ \text{is }  \nu_{j}\text{-porous on scales } R_{j+1} \ \text{to } 1.
\end{align*}
By Lemma \ref{lemma_thickening_porous_set}, the new porosity constants $\{\nu_j\}_j$ are given by
\begin{align}
    \nu_{j} = \nu -\frac{R_{1}+\dots+R_{j}}{R_{j+1}}, \ \ \text{where } \ R_{j+1}>\frac{R_{1}+\dots + R_{j}}{\nu} \ \ \text{for } j=0,1,\dots,n-1.
    \label{new_porosity_after_thickening}
\end{align}
Since the radii must also satisfy $h\leq R_{j}\leq 1$, the length of the sequences $n$ cannot be infinite and must depend on $h$. 
From the above conditions, it is clear that radii must at least grow geometrically, i.e., (in the equality case) $R_{j} = c\cdot a^{j}$ for $j=1,2,\dots,n$ where $a>1$ and $c>0$. 
In particular, the conditions are satisfied for the sequence of radii $R_j = c\cdot\left(\frac{3}{\nu}\right)^j$ for $j=1,2,\dots,n$ such that $c\cdot\left(\frac{3}{\nu}\right)^{n}\leq 1 < c\cdot\left(\frac{3}{\nu}\right)^{n+1}$. By choosing the initial radius to be as small as possible, namely $R_1 = h$, it follows that $0< n +\ln h\cdot \left(\ln\frac{3}{\nu}\right)^{-1} \leq 1$, which means $n\sim |\ln h|$.

Now we make repeated use of Corollary \ref{corollary_subaveraging_in_Fock_space_nu_porocity} to the sequence of thickened sets $\{\Omega^{(j)}\}_j$, which reveals that the integral over $\Omega$ is bounded by
\begin{align*}
    \int_{\Omega}|F(z)|^{2}e^{-\pi |z|^2}\mathrm{d}A(z) &\leq \left[\prod_{k=1}^{j}\kappa_{d}\left(\pi R_{k}^2\right)\cdot \left(1-\nu_{k-1}^{2d}\right) \right]
    \cdot\int_{\Omega^{(j)}}|F(z)|^{2}e^{-\pi|z|^2}\mathrm{d}A(z)\\ \ \ &\leq \left[\prod_{k=1}^{j}\kappa_{d}\left(\pi R_{k}^2\right)\cdot \left(1-\nu_{k-1}^{2d}\right) \right]\cdot \|F\|_{\pazocal{L}^2}^{2} \ \ \text{for } j=1,2,\dots,n.
\end{align*}
It is straightforward to verify that $\lim_{x\to 0}\kappa_{d}(x) =1$ and the derivative $\kappa_{d}'(x)>0$ for all $x>0$.
Combined with the fact that the sequence of porosity constants $\{\nu_j\}_j$ is decreasing, we conclude that the sequence $\{ \kappa_{d}\left(\pi R_{j}^2\right)\cdot \left(1-\nu_{j-1}^{2d}\right) \}_j$ is monotonically \textit{increasing}. Thus, we are interested in the index reduction $n_{0}\in \mathbb{N}$ such that
\begin{align*}
    \kappa_{d}\left(\pi R_{j}^2\right)\cdot \left(1-\nu_{j-1}^{2d}\right)< 1 \ \ \text{for } j=1,2,\dots,n-n_0.
\end{align*}
Since $\lim_{h\to0} \kappa_{d}\left(\pi R_1^2 \right)\cdot\left(1-\nu_{0}^{2d}\right) =1-\nu^{2d}<1$, for sufficiently small $h<h_0$, we have that there exists an $n_{0}$ such that the difference $n-n_{0}>0$. Next, we claim that there exists an index reduction $n_0$ which is \textit{independent} of $h<h_0$, implying $n-n_0 \sim |\ln h|$. Note that the procedure below does \textit{not} produce an optimal $n_0$, merely shows that such an $n_0$ exists.

From \eqref{new_porosity_after_thickening} with $R_j = c\left(\frac{3}{\nu}\right)^j$, we find the following lower bound for the porosities
\begin{align*}
    \nu_{j} = \nu-\left(\frac{\nu}{3}\right)^{j+1}\sum_{k=1}^{j}\left(\frac{3}{\nu}\right)^k
    \geq \frac{\nu}{2} \ \ \ \forall \ \ \ j\in\mathbb{N} \ \ \text{and } \ 0<\nu<1. 
\end{align*}
Now, let $r = r(\nu)>0$ denote the (unique) solution to $\kappa_{d}(\pi r^2)\cdot\left(1-\left(\frac{\nu}{2}\right)^{2d}\right) =1$, from which it is evident that $\kappa_{d}\left(\pi R_{j}^2\right)\cdot \left(1-\nu_{j-1}^{2d}\right)< 1$ for all $R_j <r$. Now, consider any fixed $r_0<r$, and observe that since $r_0$ is independent of $h$, and the radius $R_n(\approx 1)$ is bounded, there must exist a finite number $n_0$ also \textit{independent} of $h$ such that $R_{n-n_0} = R_n\cdot(\frac{\nu}{3})^{n_0}\leq r_0$.

In summary, by our choice of index reduction $n_0$, we obtain
\begin{align*}
    \kappa_{d}\left(\pi R_j^2\right)\cdot\left(1-\nu_{j-1}^{2d}\right) \leq \kappa_{d}(\pi r_0^2)\cdot\left(1-\left(\frac{\nu}{2}\right)^{2d}\right)=:(1-\epsilon) < 1\ \ \text{for } \ j=1,2,\dots,n-n_0. 
\end{align*}
Since $n-n_0 \sim |\ln h|$, we conclude
\begin{align*}
    \int_{\Omega}|F(z)|^{2}e^{-\pi |z|^2}\mathrm{d}A(z) \leq (1-\epsilon)^{n-n_0}\|F\|_{\pazocal{L}^2}^2\sim h^{\beta}\|F\|_{\pazocal{L}^2}^2 \ \ \text{for some } \beta >0.
\end{align*}
\end{proof}

\subsection{Fractal Uncertainty Principle for Modulation Spaces}
\label{section_FUP_modulation_spaces}
Recall that for signals $f$ and windows $\phi$ in $L^2(\mathbb{R}^d)$ the STFT $V_{\phi}f\in L^2(\mathbb{R}^{2d})$. For general $p\geq 1$, we instead consider modulation spaces, introduced in \cite{Feichtinger_1983_Modulation_spaces}. Let  $\pazocal{S}'(\mathbb{R}^d)$ denote the space of tempered distributions, and suppose $\phi:\mathbb{R}^d\to \mathbb{C}$ is a window function such that $0<\| V_{\phi}\phi\|_{L^1(\mathbb{R}^{2d})}<\infty$. Then the \textit{modulation space} $M^p(\mathbb{R}^d)$ is defined as the subspace
\begin{align*}
    M^p(\mathbb{R}^d) :=\left\{  f\in\pazocal{S}'(\mathbb{R}^d)  \ \big| \ \|V_{\phi}f\|_{L^p(\mathbb{R}^{2d})}<\infty \right\}.
\end{align*}
Depending on our particular choice of window, we induce equivalent norms on $M^p(\mathbb{R}^d)$, namely $\|f\|_{M^p}:=\|V_{\phi}f\|_{L^p(\mathbb{R}^{2d})}$. Hence, we can consider the Gaussian window $\phi_0$
and establish a connection between the modulation spaces and the Fock spaces. In fact, the Bargmann transform, defined in \eqref{Bargmann_transform_def}, extends to an isometric isomorphism $\pazocal{B}:M^p(\mathbb{R}^d)\to \pazocal{F}^p(\mathbb{C}^d)$ (see \cite{Abreu_Grochenig_Gabor_frames_2021}, \cite{Signahl_2012_Bargmann_transform_mapping_properties}). Utilizing identity \eqref{Bargmann_transform_STFT_identity}, we find that 
\begin{align*}
    \| V_{\phi_0} f\cdot \chi_{\Omega}\|_{L^{p}(\mathbb{R}^{2d})} = \| \pazocal{B}f\cdot \chi_{\Omega^*}\|_{\pazocal{L}^p},
\end{align*}
where $\Omega^{*} := \{z \in \mathbb{C}^d \ | \ \overline{z}\in \Omega \}$ denotes the complex conjugate subset. Thus, Theorem \ref{theorem_FUP_Fock_space} can be rephrased in terms of the STFT on modulation spaces.

\begin{theorem}
\label{theorem_FUP_STFT_modulation_spaces}
\textnormal{(FUP for modulation spaces $M^{p}(\mathbb{R}^{d})$)}
Let $0<h\leq 1$, and suppose $\Omega(h)\subseteq \mathbb{C}^{d}$ is an $h$-dependent family of sets which is $\nu$-porous on scales $h$ to $1$. Then for all $p\geq 1$ and all $f\in M^{p}(\mathbb{R}^d)\setminus\{0\}$ there exist constants $C, \beta >0$ only dependent on $\nu$ (and $d$) such that
\begin{align}
    \frac{\|V_{\phi_0}f\cdot \chi_{\Omega(h)}\|_{L^{p}(\mathbb{R}^{2d})}^p}{\|V_{\phi_0}f\|_{L^{p}(\mathbb{R}^{2d})}^p} \leq C h^{\beta} \ \ \forall \ \ 0<h\leq 1.&
\intertext{In particular, for $p=2$, Daubechies' time-frequency localization operator $P^{\phi_0}_{\Omega(h)}$ satisfies} 
    \| P^{\phi_0}_{\Omega(h)}\|_{\textnormal{op}} \leq C h^{\beta} \ \ \forall \ \ 0<h\leq 1.&
\end{align}
\end{theorem}
Another noteworthy rephrasement is result \eqref{Fock_space_nyquist_density}, which in terms of the STFT reads
\begin{align}
    \| V_{\phi_0}f\cdot \chi_{\Omega}\|_{L^p(\mathbb{R}^{2d})}^p \leq 2^{-d}&\frac{\rho(\Omega,R)}{1-e^{-\frac{1}{2}\pi R^2}\sum_{j=0}^{d-1}\frac{\left(\frac{1}{2}\pi R^2\right)^j}{j!}} \ \| V_{\phi_0}f\|_{L^p(\mathbb{R}^{2d})}^p \ \ \forall \ \ R>0.
    \label{Nyquist_density_STFT_inequality}
\end{align}
\begin{remark}
For $d=1$, in \cite{ABREU2021103032_donoho_logan_large_sieve} Theorem 3, there is a similar but more general result, valid for \textit{all} Hermite windows $\{h_j\}_{j}$, namely 
\begin{align*}
    \| V_{h_j}f\cdot\chi_{\Omega}\|_{L^{p}(\mathbb{R}^{2})}^{p} \leq \frac{\rho(\Omega,R)}{1-e^{-\pi R^2}P_j(\pi R^2)}\| V_{h_j}f\|_{L^{p}(\mathbb{R}^{2})}^{p} \ \ \forall \ \ p\geq 1 \ \text{and} \ R>0,
\end{align*}
where $P_j$ is a specified polynomial of degree $2j$. In particular, for $j=0$, we have that $h_0 = \phi_0$ and $P_0 \equiv 1$ so that the above inequality reduces to
\begin{align}
    \| V_{\phi_0}f\cdot\chi_{\Omega}\|_{L^{p}(\mathbb{R}^{2})}^{p} \leq \frac{\rho(\Omega,R)}{1-e^{-\pi R^2}}\| V_{\phi_0}f\|_{L^{p}(\mathbb{R}^{2})}^{p} \ \ \forall \ \ p\geq 1 \ \text{and} \ R>0.
    \label{abreu_speckbacher_donoho_logan_large_sieve_gaussian_result}
\end{align}
By close inspection, for $d=1$ result \eqref{Nyquist_density_STFT_inequality} turns out to be an \textit{improvement} of \eqref{abreu_speckbacher_donoho_logan_large_sieve_gaussian_result}.\footnote{In \cite{HelgeKnutsen_Daubechies_2020} (section 3.2 page 10) there is a similar comparison for a specific set $\Omega$, where $\rho(\Omega,R)$ is known, with the claim that the upper bound obtained in \cite{HelgeKnutsen_Daubechies_2020} is also an improvement of \eqref{abreu_speckbacher_donoho_logan_large_sieve_gaussian_result}. By closer examination, this claim is \textit{incorrect}, and the upper bound for the localization operator is in fact the \textit{same} as result \eqref{abreu_speckbacher_donoho_logan_large_sieve_gaussian_result}, i.e., the special case $h_0 = \phi_0$ of Theorem 3 in \cite{ABREU2021103032_donoho_logan_large_sieve}.} 
\end{remark}

\section{Density of Cantor sets}
\label{section_density_of_cantor_sets}
In this section we consider the Cantor sets specifically and show, under certain conditions, how the FUP in Theorem \ref{theorem_FUP_STFT_modulation_spaces} (or equivalently Theorem \ref{theorem_FUP_Fock_space}) can be refined for this family of fractal sets. While the general FUP is formulated in terms of a continuous parameter $h\to 0^{+}$, it is more convenient to consider discrete iterations $n\to \infty$ when working with the Cantor sets. The relation between $h$ and $n$ is made clear by Lemma \ref{lemma_Cantor_set_porosity}, where for the Cantor iterate $\pazocal{C}_n(M,\mathcal{A},L)$, we have that
\begin{align*}
    h \sim L M^{-n}.
\end{align*}
From the above relation, we observe that $h$ is of the same order of magnitude as the intervals that make up the Cantor iterate. 
Furthermore, similarly to the condition imposed in \cite{DyatlovFUP2019}, we consider $L\sim h^{-1}$. Thus, we have that $L$ is dependent on the iterates $n$ so that $L(n) \sim M^{\frac{n}{2}}$. More precisely, we specify the multiplicative constants of the asymptotes so that the interval condition now reads:
\begin{definition}
Let the interval length $L$ be a function $\mathbb{N}\to \mathbb{R}_{+}$. The interval length $L$ satisfies condition $(I_M)$ with constants $0<c_1\leq c_2< \infty$ if
\begin{align}
    c_1 M^{\frac{n}{2}}\leq L(n) \leq c_2 M^{\frac{n}{2}} \ \ \text{for } \ n=0,1,2,\dots
    \label{Cantor_set_interval_length_condition}
\end{align}
\end{definition}
For the $2d$-dimensional radially symmetric Cantor iterate $\mathcal{C}_n^d(R,M,\mathcal{A})$, we adjust the above condition to the radius $R>0$, similar to condition (1.3) in \cite{HelgeKnutsen_Daubechies_2020} and (1.4) in \cite{HelgeKnutsen_2022}.
\begin{definition}
Let the radius $R$ be a function $\mathbb{N}\to \mathbb{R}_{+}$. The radius $R$ satisfies condition $(D_M)$ with constants $0<c_1\leq c_2< \infty$ if
\begin{align}
    c_1 M^{\frac{n}{2}}\leq R^{2d}(n)\leq c_2 M^{\frac{n}{2}}, \ \ \text{for } \ n=0,1,2,\dots
    \label{radius_condition}
\end{align}
\end{definition}

With these conditions, we present an FUP for the two multidimensional Cantor set constructions from section \ref{section_porous_sets_Cantor_sets}.

\begin{theorem}
\label{theorem_FUP_for_Cantor_sets}
\textnormal{(FUP for Cantor sets)} We consider the two constructions separately:
\begin{enumerate}[label =  \textnormal{(\roman*)},itemsep=0.4ex, before={\everymath{\displaystyle}}]
    \item Let $\mathcal{C}_n^{2d}(R,M,\mathcal{A})$ denote the $2d$-dimensional radially symmetric Cantor iterate, and suppose the radius $R = R(n)>0$ satisfies condition $(D_M)$ with constants $c_1\leq c_2$. Then there exists a positive constant $\gamma$ only dependent on $M$, $|\mathcal{A}|$ and $c_1, c_2$ such that for all $f\in M^p(\mathbb{R}^{d})$ with $p\geq 1$,
\begin{align}
    \| V_{\phi_0}f\cdot \chi_{\mathcal{C}^{2d}_n(R(n),M,\mathcal{A})}\|_{L^{p}(\mathbb{R}^{2d})}^p \leq \gamma \cdot \left(\frac{|\mathcal{A}|}{M}\right)^{\frac{n}{2}} \| V_{\phi_0}f\|_{L^{p}(\mathbb{R}^{2d})}^p \ \ \text{for } \ n=0,1,2,\dots
    \label{FUP_radial_Cantor_set}
\end{align}
    \item Let $\Omega:=\pazocal{C}_{n_1}(L_1,M_1,\mathcal{A}_1)\times\dots\times \pazocal{C}_{n_{2d}}(L_{2d},M_{2d},\mathcal{A}_{2d})$ denote the $2d$-dimensional Cartesian product of Cantor iterates, and suppose the interval lengths $L_j = L_j(n_j)$ satisfy condition $(I_{M_j})$ with the same constants $c_1\leq c_2$. Then there exists a positive constant $\gamma$ only dependent on $\{(M_j,|\mathcal{A}_j|)\}_j$ and $c_1, c_2$ such that for all $f\in M^{p}(\mathbb{R}^{d})$ with $p\geq 1$ and all iterations $\{n_j\}_j\in \mathbb{N}_{0}^{2d}$, 
\begin{align*}
    \| V_{\phi_0} f\cdot \chi_{\Omega}\|_{L^{p}(\mathbb{R}^{2d})}^{p} \leq \gamma\cdot \left[\prod_{j=1}^{2d}\left(\frac{|\mathcal{A}_j|}{M_j}\right)^{n_j/2}\right] \| V_{\phi_0}f\|_{L^{p}(\mathbb{R}^{2d})}^{p}.
\end{align*}
In particular, when all iterates are equal to say $\pazocal{C}_n(L(n),M,\mathcal{A})$, we find that 
\begin{align}
    \| V_{\phi_0}f\cdot \chi_{\pazocal{C}_n(L,M,\mathcal{A})^{2d}}\|_{L^p(\mathbb{R}^{2d})}^p\leq \gamma\cdot\left(\frac{|\mathcal{A}|}{M}\right)^{n\cdot d} \| V_{\phi_0}f \|_{L^p(\mathbb{R}^{2d})}^p \ \ \text{for }\ n=0,1,2\dots
\end{align}
\end{enumerate}
\end{theorem} 

\begin{remark}
For localization on the radially symmetric Cantor set with $d=1$ and $p=2$, we retrieve the same upper bound asymptote as in \cite{HelgeKnutsen_2022} Theorem 3.3. In fact, Theorem 3.3 reveals that the general asymptote is \textit{optimal} in the sense that there are alphabets where the asymptote is \textit{precise}. However, the same theorem also state that there are alphabets where the asymptote is \textit{not} precise.
\end{remark}

The proof of Theorem \ref{theorem_FUP_for_Cantor_sets} is based on inequality \eqref{Nyquist_density_STFT_inequality}, combined with estimates of the maximal Nyquist density of the multidimensional Cantor iterates. 
To begin with, we consider the 1-dimensional $n$-iterate $\pazocal{C}_n(M,\mathcal{A}):=\pazocal{C}_{n}(1,M,\mathcal{A})$ based in $[0,1]$, to which we associate the so-called \textit{Cantor function} $\pazocal{G}_{n, M,\mathcal{A}}:\mathbb{R}\to [0,1]$, given by
\begin{align}
    \pazocal{G}_{n,M,\mathcal{A}}(x) := |\pazocal{C}_n(M,\mathcal{A})|^{-1}
    \begin{cases}
    0, \ \ x\leq 0,\\
    |\pazocal{C}_n(M,\mathcal{A})\cap[0,x]|, \ \ x>0.
    \end{cases}
    \label{Cantor_function}
\end{align}
For the iterates based in $[0,L]$, we consider the dilated Cantor function $\pazocal{G}_{n,M,\mathcal{A}}(\cdot\  L^{-1})$. The Cantor function is a useful concept as the difference $\pazocal{G}_{n,M,\mathcal{A}}(y)-\pazocal{G}_{n,M,\mathcal{A}}(x)$ measures the \textit{portion} of the $n$-iterate $\pazocal{C}_n(M,\mathcal{A})$ contained in the interval $[x,y]$. By definition, the Cantor function is said to be \textit{subadditive} if the difference is bounded by $\pazocal{G}_{n,M,\mathcal{A}}(y-x)$. While this is the case for the mid-third Cantor set (see \cite{JosefDobos_1996}), we cannot guarantee subadditivity with an arbitrary alphabet. Nonetheless, for our purpose, we only require a \textit{weaker} version utilizing the canonical alphabet $\overline{\mathcal{A}}:=\{0,1,\dots,|\mathcal{A}|\}$ (see \cite{HelgeKnutsen_2022} Appendix A).

\begin{lemma}
\label{lemma_Cantor_function_weak_subadditivity_property}
Let $\pazocal{G}_{n,M,\mathcal{A}}$ denote the Cantor-function, defined in \eqref{Cantor_function}. Then for any $x\leq y$,
\begin{align}
    \pazocal{G}_{n,M,\mathcal{A}}(y)-\pazocal{G}_{n,M,\mathcal{A}}(x) \leq \pazocal{G}_{n,M,\overline{\mathcal{A}}}(y-x).
    \label{weak_subadditivity_equation}
\end{align}
\end{lemma}

Re-scaling to Cantor iterates based $[0,L]$, we find, by the above lemma, that the maximal Nyquist density is bounded by
\begin{align*}
    \rho\big(\pazocal{C}_n(L,M,\mathcal{A}),x\big) = \max_{a\in \mathbb{R}}|\pazocal{C}_{n}(L,M,\mathcal{A})\cap[a,a+x]|\leq |\pazocal{C}_n(L,M,\overline{\mathcal{A}})\cap[0,x]|.
\end{align*}
In the next lemma, we enforce condition \eqref{Cantor_set_interval_length_condition}, which yields a more explicit estimate for the upper bound.

\begin{lemma}
\label{lemma_cantor_function_decay}
Suppose that the alphabet $\mathcal{A}$ is a proper subset of $\{0,1,\dots,M-1\}$, and suppose that the length $L(n)>0$ satisfies condition $(I_M)$ with constants $0<c_1\leq c_2 < \infty$. Then for any fixed $x> 0$, there exists a finite constant $\gamma>0$ dependent only on $x, c_1, c_2$ such that
\begin{align}
    |\pazocal{C}_n(L,M,\mathcal{A})\cap[0,x]| =\pazocal{G}_{n,M,\mathcal{A}}\big(xL^{-1}\big)\left(\frac{|\mathcal{A}|}{M}\right)^{n}L\leq \gamma\cdot \left(\frac{|\mathcal{A}|}{M}\right)^{\frac{n}{2}}\to 0 \ \text{as } n\to \infty.
    \label{Cantor_set_density}
\end{align}

\begin{proof}
Fix some positive integer $n_0<n$, and observe that for coefficients $a_j\in \mathcal{A}$, the sum $\sum_{j=0}^{n-1}a_j M^{j-n}<M^{n_0-n}$ only if $a_j =0$ for $j\geq n_0$. Thus, the Cantor function is bounded by 
\begin{align*}
    \pazocal{G}_{n,M,\mathcal{A}}(x) \leq |\mathcal{A}|^{n_0-n} \ \ \forall \ \ x<M^{n_0-n}.
\end{align*}
Evidently, we also have that $xM^{-m}<M^{n_0+m-(n-2m)}$, which, by the above estimate, yields
\begin{align}
    \pazocal{G}_{n,M,\mathcal{A}}(xM^{-m}) \leq |\mathcal{A}|^{n_0+m-(n+2m)}=|\mathcal{A}|^{-m}|\mathcal{A}|^{n_0-n} \ \ \text{for }\  m=0,1,2,\dots
    \label{Cantor_function_bound}
\end{align}
Now, suppose $x$ is \textit{any} fixed number, to which there exists a threshold iterate $N\geq 0$ so that $xM^{-\frac{N}{2}}<M^{-1}$, and we are in a position to apply \eqref{Cantor_function_bound}. From here, we deduce that there exists a constant only dependent on $x$, say $\kappa = \kappa(x)$ such that 
\begin{align*}
    \pazocal{G}_{n,M,\mathcal{A}}(xM^{-\frac{n}{2}}) \leq \kappa(x) |\mathcal{A}|^{-\frac{n}{2}} \ \ \text{for }\ n=0,1,2,\dots
\end{align*}
Enforce the lower bound condition $L(n)\geq c_1 M^{\frac{n}{2}}$. By monotonicity of the Cantor function, 
\begin{align*}
    \pazocal{G}_{n,M,\mathcal{A}}(xL^{-1})\leq \pazocal{G}_{n,M,\mathcal{A}}(x c_{1}^{-1} M^{-\frac{n}{2}})\leq \kappa(x c_1^{-1})|\pazocal{A}|^{-\frac{n}{2}}.
\end{align*}
Finally, enforce the upper bound condition $L(n)\leq c_2 M^{\frac{n}{2}}$, and inequality \eqref{Cantor_set_density} follows with constant $\gamma = \kappa(xc_1^{-1})c_2$. Since $0<c_1 \leq c_2 < \infty$, the constant $\gamma$ is indeed finite.
\end{proof}
\end{lemma}

We are now ready to estimate upper bounds for the maximal Nyquist densities for the multidimensional Cantor iterates, which, combined with \eqref{Nyquist_density_STFT_inequality}, proves Theorem \ref{theorem_FUP_for_Cantor_sets}.

\begin{lemma}
\label{lemma_nyquist_density_Cantor_sets}
\textnormal{(Nyquist density of Cantor sets)} Let $r>0$ be fixed.
\begin{enumerate}[label =  \textnormal{(\roman*)},itemsep=0.4ex, before={\everymath{\displaystyle}}]
\item Let $\mathcal{C}_{n}^{2d}(R,M,\mathcal{A})$ denote the radial Cantor iterate, and suppose the radius $R(n)>0$ satisfies condition $(D_M)$ with constants $c_1 \leq c_2$. Then there exist a constant $\gamma_r>0$ only dependent on $M, |\mathcal{A}|$ and $c_1,c_2$ such that the maximal Nyquist density
\begin{align*}
    \rho\left(\mathcal{C}_{n}^{2d}(R,M,\mathcal{A}), r\right) \leq \gamma_r\cdot \left(\frac{|\mathcal{A}}{M}\right)^{\frac{n}{2}} \ \ \text{for } \ n=0,1,2,\dots
\end{align*}
\item Let $\Omega:=\pazocal{C}_{n_1}(L_1,M_1,\mathcal{A}_1)\times \dots \times \pazocal{C}_{n_{2d}}(L_{2d},M_{2d},\mathcal{A}_{2d})$ denote the Cartesian product, and suppose the interval lengths $L_j(n_j)$ satisfy condition $(I_{M_j})$ with the same constants $c_1\leq c_2$. Then there exist a constant $\gamma_r>0$ only dependent on $\{(M_j, |\mathcal{A}_j|)\}_j$ and $c_1,c_2$ such that the maximal Nyquist density is bounded by
\begin{align*}
    \rho\left(\Omega,r\right)\leq \gamma_r \cdot \prod_{j=1}^{2d}\left(\frac{|\mathcal{A}_j|}{M_j}\right)^{n_j/2} \ \ \text{for } \ n=0,1,2,\dots \ \ \ \ \ \ \ \
\end{align*}
\end{enumerate}
\begin{proof}
Without loss of generality, we can assume that $r\geq 1$. For part (i), fix a cutoff value $N\gg 1$ and distinguish between two cases for $B_{r}(a)\in \mathbb{R}^{2d}$, (1) $|a|\leq Nr$ and (2) $|a|>Nr$:

For case (1), $B_{r}(a)\subseteq B_{(N+1)r}(0)$ so that $|\mathcal{C}^{2d}_n(\dots)\cap B_{r}(a)|\leq |\mathcal{C}^{2d}_n(\dots)\cap B_{(N+1)r}(0)| $. By definition of the radial Cantor iterate, it follows that
\begin{align*}
    \left|\mathcal{C}^{2d}_{n}(R,M,\mathcal{A})\cap B_{r}(a)\right|&\leq \frac{\pi^d}{d!}\left|\pazocal{C}_{n}(R^{2d}, M,\mathcal{A})\cap \left[0,(N+1)^{2d}r^{2d}\right]\right|\\
    &\leq \gamma_{1}\cdot \left(\frac{|\mathcal{A}|}{M}\right)^{\frac{n}{2}} \ \ \text{(by Lemma \ref{lemma_cantor_function_decay})}
\end{align*}
for some constant $\gamma_{1}>0$.

For case (2), $B_{r}(a)\subseteq A_{r}(|a|):=\{z\in \mathbb{R}^{2d} \ | \ 0<|a|-r\leq |z|\leq |a|+r \}$. For the annulus $A_{r}(|a|)$ with $r< |a|$, we have again, by definition of the radial Cantor iterate, that
\begin{align}
    \big|\mathcal{C}^{2d}_{n}(R,M,\mathcal{A})\cap A_{r}(|a|)\big| = \frac{\pi^d}{d!}\big|\pazocal{C}_{n}(R^{2d}, M,\mathcal{A})\cap\left[(|a|-r)^{2d},(|a|+r)^{2d}\right]\big|&\nonumber\\
    \leq \frac{\pi^d}{d!}\big|\pazocal{C}_{n}(R^{2d}, M,\overline{\mathcal{A}})\cap\left[0,(|a|+r)^{2d}-(|a|-r)^{2d}\right]\big| \ \ \text{(by Lemma \ref{lemma_Cantor_function_weak_subadditivity_property})}.&
    \label{case_2_radial_Cantor_iterate_intermediate_calculation}
\end{align}
Since $|a|>1$ and since the leading term $|a|^{2d}$ cancels, it follows that
\begin{align*}
    &(|a|+r)^{2d}-(|a|-r)^{2d} \leq |a|^{2d-1}\left((1+r)^{2d}-(1-r)^{2d}\right).\\
\intertext{Furthermore, by subadditivity, the Cantor function $\pazocal{G}_{n,M,\overline{\mathcal{A}}}(m\cdot x)\leq m\cdot\pazocal{G}_{n,M,\overline{\mathcal{A}}}(x)$ for any $m\in \mathbb{N}$. In general, this means}
    &\pazocal{G}_{n,M,\overline{\mathcal{A}}}(m\cdot x) \leq (m+1)\cdot\pazocal{G}_{n,M,\overline{\mathcal{A}}}(x) \ \ \forall \ \ m>0.
\end{align*}
With these observations inequality \eqref{case_2_radial_Cantor_iterate_intermediate_calculation} simplifies to 
\begin{align}
    \big|\mathcal{C}^{2d}_{n}(R,M,\mathcal{A})\cap A_{r}(|a|)\big| &\leq \frac{\pi^d}{d!}\left(|a|^{2d-1}+1\right) \big|\pazocal{C}_{n}(R^{2d}, M,\overline{\mathcal{A}})\cap\left[0,(1+r)^{2d}-(1-r)^{2d}\right]\big|\nonumber\\
    &\leq \gamma_{2} \cdot \left(|a|^{2d-1}+1\right) \left(\frac{|\mathcal{A}|}{M}\right)^{\frac{n}{2}} \ \ \text{(by Lemma \ref{lemma_cantor_function_decay})}
    \label{case_2_radial_Cantor_iterate_annulus_bound}
\end{align}
for some constant $\gamma_2>0$ \textit{independent} of $|a|$. Evidently, the right-hand side of the above inequality is unbounded in terms of $|a|$. The ball $B_{r}(a)$, however, only represents a fraction of the annulus $A_{r}(|a|)$, which warrants a closer comparison between $|\mathcal{C}^{2d}_{n}(\dots)\cap A_{r}(|a|)|$ and $|\mathcal{C}^{2d}_{n}(\dots)\cap B_{r}(a)|$.
Let $\partial B_{\eta}:=\{z\in\mathbb{R}^{2d} \ | \ |z|=\eta\}$ denote the $(2d-1)$-dimensional sphere of radius $\eta>0$ with associated surface area $|\partial B_{\eta}|$. We consider the optimal surface quotient
\begin{align*}
    \max \left\{\frac{|\partial B_{\eta}\cap B_{r}(a)|}{|\partial B_{\eta}|} \ \Big| \ \eta \in \big[|a|-r,|a|+r\big] \right\}\leq \alpha\cdot\left(\frac{r}{|a|}\right)^{2d-1}
\end{align*}
for some constant $\alpha>0$ only dependent on $N\gg 1$ and $d$. Hence, 
\begin{align*}
    |\mathcal{C}^{2d}_{n}(R,M,\mathcal{A})\cap B_{r}(a)| \leq \alpha\cdot \left(\frac{r}{|a|}\right)^{2d-1}|\mathcal{C}^{2d}_{n}(R,M,\mathcal{A})\cap A_{r}(|a|)|,
\end{align*}
which, combined with \eqref{case_2_radial_Cantor_iterate_annulus_bound}, yields the desired conclusion of part (i). 

Part (ii) turns out to be much simpler. Since any ball $B_{r}(a)$ is contained in some shifted hypercube $a + [-r,r]^{2d}$, we consider the maximal Nyquist density $\rho\big(\pazocal{C}_{n_j}(L_j,M_j,\mathcal{A}_j),r)$ in each direction $j=1,2,\dots,2d$. In total, we obtain 
\begin{align*}
    \rho(\Omega,r) \leq \prod_{j=1}^{2d}\rho\big(\pazocal{C}_{n_j}(L_j,M_j,\mathcal{A}_j),r),
\end{align*}
where the desired conclusion follows once we apply the 1-dimensional result Lemma \ref{lemma_cantor_function_decay}.
\end{proof}
\end{lemma}

\section{Fractal Uncertainty Principle for Gabor Multipliers}
\label{section_FUP_Gabor_multipliers}
In this section we present one simple translation of Theorem \ref{theorem_FUP_Fock_space} to Gabor multipliers. Specifically, since the previous theorems are based on the Gaussian window $\phi_0$, we proceed with Gabor multipliers on the form \eqref{Gabor_multiplier_subset_formula} with $\phi_0$ as generating function, i.e.,
\begin{align}
    \mathcal{G}_{\Lambda, \Omega}^{\phi_0}f=|A_{\Lambda}|\sum_{\lambda\in\Lambda_{\Omega}}\langle f, \pi(\lambda)\phi_0\rangle \pi(\lambda)\phi_0
    \label{Gaussian_Gabor_multiplier_def}
\end{align}
for some subset $\Omega\subseteq\mathbb{R}^{2d}$ and lattice $\Lambda$. Furthermore, since we consider an $h$-dependent family $\Omega(h)$ of sets for the FUP, we also let the lattice $\Lambda$ depend on the parameter $0<h\leq 1$, meaning, we let $\Lambda(h)$ become sufficiently dense so to capture the fractal details of $\Omega(h)$. In particular, we consider the following condition: 
\begin{definition}
We say that an $h$-dependent family of lattices $\Lambda(h)\subseteq\mathbb{R}^{2d}$ for $0<h\leq 1$ satisfies condition $(H)$ with constant $L>0$ if the fundamental region, $A_{\Lambda(h)}$, satisfies the inclusion
\begin{align*}
     A_{\Lambda(h)} \subseteq B_{h L}(0) \ \ \forall \ \ 0<h\leq 1.
\end{align*}
\end{definition}

Utilizing condition ($H$), we formulate the FUP for Gaussian Gabor multipliers: 

\begin{theorem}
\label{theorem_FUP_Gabor_multipliers}
\textnormal{(FUP Gaussian Gabor multipliers)} Let $0<h\leq 1$,  and suppose $\Omega(h)\subseteq \mathbb{R}^{2d}$ is an $h$-dependent family of sets which is $\nu$-porous on scales $h$ to $1$. Let $\pazocal{G}_{\Lambda(h),\Omega(h)}^{\phi_0}:L^2(\mathbb{R}^d)\to L^2(\mathbb{R}^d)$ denote the Gaussian Gabor multiplier defined in \eqref{Gaussian_Gabor_multiplier_def}, whose lattice $\Lambda(h)$ satisfies condition $(H)$ with constant $L>0$. Then there exists constants $C, \beta>0$ only dependent on $\nu, L$ (and $d$) such that the operator norm is bounded by
\begin{align*}
    \| \mathcal{G}_{\Lambda(h),\Omega(h)}^{\phi_0}\|_{\text{op}}\leq C h^{\beta} \ \ \forall \ \ 0<h\leq 1. 
\end{align*}
\end{theorem}

First, we reformulate the problem to an estimate in the Fock space. 

\begin{lemma}
\label{lemma_Gabor_multiplier_Bargmann_transform}
Let $\Omega^{*}:=\{z \in \mathbb{C}^d \ | \ \overline{z}\in\Omega\}$ denote the complex conjugate set of $\Omega \subseteq\in\mathbb{R}^{2d}$. The operator norm of the Gaussian Gabor multiplier $\mathcal{G}_{\Lambda,\Omega}^{\phi_0}:L^2(\mathbb{R}^{d})\to L^2(\mathbb{R}^{d})$ is given by 
\begin{align*}
    \| \mathcal{G}_{\Lambda,\Omega}^{\phi_{0}} \|_{\text{op}} = \sup_{F\in \pazocal{F}^2(\mathbb{C}^d),\  \|F\|_{\pazocal{L}^2}=1}|A_{\Lambda}|\sum_{\lambda \in \Lambda_{\Omega^{*}}}|F(\lambda)|^{2}e^{-\pi|\lambda|^2}.
\end{align*}
\begin{proof}
By Cauchy-Schwarz' inequality on $\ell^2$-sequences, the operator norm is given by 
\begin{align*}
    \| \mathcal{G}_{\Lambda, \Omega}^{\phi_0}\|_{\text{op}} = \sup_{\|f\|_{2}=1}\langle \mathcal{G}_{\Lambda, \Omega}^{\phi_0} f,f\rangle = \sup_{\|f\|_{2}=1}|A_{\Lambda}|\sum_{\lambda\in\Lambda_{\Omega}}|\langle f,\pi(\lambda)\phi_0\rangle|^2.
\end{align*}
By identity \eqref{Bargmann_transform_STFT_identity}, each term can be expressed in terms of the Bargmann transform, namely $|\langle f, \pi(\overline{\lambda})\phi_{0}\rangle|^2 = |\pazocal{B}f(\lambda)|^2 e^{-\pi|\lambda|^2}$, and since the Bargmann transform $\pazocal{B}:L^2(\mathbb{R}^d)\to \pazocal{F}^2(\mathbb{C}^d)$ is an isometry onto the Fock space, the result follows.
\end{proof}
\end{lemma}

\begin{proof}
(Theorem \ref{theorem_FUP_Gabor_multipliers}) By Lemma \ref{lemma_Gabor_multiplier_Bargmann_transform}, we consider the sum $|A_{\Lambda}|\sum_{\lambda}|F(\lambda)|^{2} e^{-\pi|\lambda|^2}$ for any normalized $F\in\pazocal{F}^2(\mathbb{C}^{d})$. By the subaveraging property in Lemma \ref{lemma_subaveraging_in_Fock_space}, for every $R>0$
\begin{align*}
    \sum_{\lambda}|F(\lambda)|^{2} e^{-\pi|\lambda|^2}&\leq \left(1-e^{-\pi R^2}\sum_{j=0}^{d-1}\frac{(\pi R^2)^j}{j!}\right)^{-1}\sum_{\lambda}\int_{B_{R}(\lambda)}|F(\xi)|^2 e^{-\pi|\xi|^2}\mathrm{d}A(\xi)\\
    &\leq \left(1-e^{-\pi R^2}\sum_{j=0}^{d-1}\frac{(\pi R^2)^j}{j!}\right)^{-1} \gamma(R,\Lambda)\cdot \int_{\bigcup_{\lambda}B_{R}(\lambda)}|F(\xi)|^2 e^{-\pi|\xi|^2}\mathrm{d}A(\xi)
\end{align*}
where 
\begin{align*}
    \gamma(R,\Lambda) := \sup_{z\in A_{\Lambda}}\mathbf{card} \big\{ \lambda \in \Lambda \ \big| \ |B_{R}(z)\cap(A_{\Lambda}+\lambda)|>0\big\}
\end{align*}
takes into account the possible \textit{overlap} between the balls $B_{R}(\lambda)$. 
In particular, for $R=h$ and by the condition $A_{\Lambda(h)}\subseteq B_{Lh}(0)$, there clearly exists some finite constant $C_{d}>0$ dependent on $d, L$ and otherwise independent of the lattice structure so that 
\begin{align}
    &\gamma(h,\Lambda(h)) \leq C_{d} \frac{|B_{h}(0)|}{|A_{\Lambda(h)}|} \ \ \forall 
    \ \ 0<h\leq 1.\nonumber
\end{align} 
In total, we obtain
\begin{align*}
    |A_{\Lambda(h)}|\sum_{\lambda\in \Lambda_{\Omega^{*}}(h)}|F(\lambda)|^2 e^{-\pi|\lambda|^2}\leq C_{d}\cdot \kappa_{d}(\pi h^{2}) \int_{\bigcup_{\lambda\in\Lambda_{\Omega^{*}(h)}}B_{h}(\lambda)}|F(\xi)|^2 e^{-\pi|\xi|^2}\mathrm{d}A(\xi),
\end{align*}
where $\kappa_{d}(x)$ was defined in \eqref{thickening_constant_kappa}. Since $\kappa_{d}(x)$ is continuous for $x>0$ and $\lim_{x\to 0}\kappa_{d}(x) =1$, this factor is simply absorbed by the multiplicative constant $C>0$ of the theorem. 

Now, by definition \eqref{lattice_restriction_to_set} for the subset $\Lambda_{\Omega^*}(h)\subseteq \Lambda(h)$ and since $A_{\Lambda(h)}\subseteq B_{h L}(0)$, the union $\bigcup_{\lambda\in\Lambda_{\Omega^*}(h)}B_{h}(\lambda)$ must be contained in the $r$-thickened set $\Omega_{r}(h)$ for $r:=(1+L)h$. By Lemma \ref{lemma_thickening_porous_set}, the thickened set $\Omega_{r}(h)$ is itself $\frac{\nu}{2}$-porous on scales $h_1:=\frac{4r}{\nu}h$ to $1$. Thus, we simply apply the FUP for the Fock space in Theorem \ref{theorem_FUP_Fock_space} on the family of sets $\Omega_{hr}(h)$ as $h_1\to 0$, from which the statement follows. 
\end{proof}

\begin{remark}
The question immediately arises whether the Gabor \textit{systems} $\{\pi(\lambda)\phi_0\}_{\lambda\in \Lambda(h)}$ with $\Lambda(h)$ satisfying condition ($H$) actually are Gabor \textit{frames} for all $0<h\leq 1$ or for sufficiently small $h$. In $d=1$, by the results obtained independently by Lyubarskii \cite{Lyubarskii_1992} and Seip and Wallsten \cite{Seip_1992_prtI}, \cite{Seip_Wallsten_1992_prtII}, we have a simple density criterion. Namely, that the system $\{\pi(\lambda)\phi_{0}\}_{\lambda\in\Lambda}$ is a Gabor frame if  and only if the fundamental region satisfies $|A_{\Lambda}|<1$. In higher dimensions $d>1$, characterizations of lattices $\Lambda$ that yield Gaussian Gabor frames becomes much more intricate, which has been studied in \cite{Grochenig_2011_Gaussian_Gabor_frames}, \cite{Grochenig_Lyubarskii_2020}, \cite{Luef_Wang_2021}. In particular, the density criterion does not translate, e.g., as shown in \cite{Luef_Wang_2021} Theorem 1.5, the lattice $\Lambda = (\mathbb{Z}\times \frac{1}{2}\mathbb{Z})^2$ does not generate a Gaussian Gabor frame even though $|A_{\Lambda}|=\frac{1}{4}<1$. Further, a sufficient condition is formulated in \cite{Luef_Wang_2021} Theorem 1.2 combining a density criterion with the notion of \textit{transcendental} lattices, which, as remarked by the authors, represents a large family of lattices in $\mathbb{C}^{d}$. Nonetheless, our estimates of the Gabor multiplier are not reliant on the operator being associated to a Gabor frame. 
\end{remark} 

We conclude this section with a simple example of Gabor multipliers based on Cantor sets, which also illustrates an alternative approach to choosing the lattice restriction $\Lambda_{\Omega}$. 

\begin{example}
(Cantor set) For simplicity, we let $d=1$ and consider the symmetric Cartesian product of Cantor iterates $\Omega:=\pazocal{C}_{n}(L,M,\mathcal{A})\times \pazocal{C}_{n}(L,M,\mathcal{A})$. For this case one obvious choice of lattices are square lattices $\Lambda = a\left(\mathbb{Z}\times \mathbb{Z}\right)$ with density $a\sim L M^{-n}$. Unsurprisingly, the restriction $\Lambda_{\Omega}$ closely resembles the Cartesian product of scaled \textit{discrete} Cantor iterates $\Omega^{(d)}:=[L\cdot\pazocal{C}_{n}^{(d)}(M,\mathcal{A})]^2$. Thus, for Gabor multipliers localizing on such Cartesian products, it seems more natural to consider sampled points directly from the already available discrete set, i.e., we consider the operator
\begin{align*}
    &\mathcal{G}_{\phi_0}(\Omega^{(d)})f:=(L M^{-n})^2\sum_{\lambda \in \left[L\cdot\pazocal{C}^{(d)}(M,\mathcal{A})\right]^2}\langle f, \pi(\lambda)\phi_{0}\rangle \pi(\lambda)\phi_{0}.
\intertext{In order to estimate the operator norm, we follow the same procedure as in the proof of Theorem \ref{theorem_FUP_Gabor_multipliers}. After one $\frac{1}{2}LM^{-n}$-thickening and then utilizing inequality \eqref{Nyquist_density_STFT_inequality}, we obtain for all $f\in L^2(\mathbb{R})$ and $r>0$ that}
    &\| \mathcal{G}_{\phi_0}(\Omega^{(d)})f\|_{2} \leq \frac{\kappa_1\big(\frac{\pi}{4}(L M^{-n})^2\big)}{1-e^{\frac{1}{2}\pi r^2}} \rho\left(\Omega^{(d)}+B_{\frac{1}{2}LM^{-n}}(0),r\right)\|f\|_2^2. 
\end{align*}
By definition of the discrete and "continuous" Cantor iterate, the Nyquist density must satisfy $\rho(\Omega^{(d)}+B_{\frac{1}{2}LM^{-n}}(0),r) \leq \rho(\Omega,r)$, so the operator norm is in turn bounded by
\begin{align*}
    \| \mathcal{G}_{\phi_0}(\Omega^{(d)})\|_{\text{op}} \leq \frac{\kappa_1\big(\frac{\pi}{4}(L M^{-n})^{2}\big)}{1-e^{\frac{1}{2}\pi r^2}}\rho(\Omega,r).
\end{align*}
If we now suppose the length $L$ depends on the iterates $n$ according to condition ($I_M$), we simply apply the estimate for the Nyquist density in Lemma \ref{lemma_nyquist_density_Cantor_sets} to retrieve the same asymptotic estimate for the operator norm as in Theorem \ref{theorem_FUP_for_Cantor_sets}. 
\end{example}

\appendix 

\section{Complex interpolation in Fock space}
\label{appendix_complex_interpolation}

We prove the following inequality between the Fock spaces $\pazocal{F}^1(\mathbb{C}^d)$ and $\pazocal{F}^p(\mathbb{C}^d)$ for $p\geq 1$. 
\begin{lemma}
\label{lemma_complex_interpolation_characteristic_function}
For any measurable subset $\Omega\subseteq \mathbb{C}^{d}$ and any $p\geq 1$, we have the inequality
\begin{align*}
    \sup_{F\in \pazocal{F}^p(\mathbb{C}^d)\setminus \{0\}} \frac{\|F\cdot \chi_{\Omega}\|_{\pazocal{L}^p}^p}{\|F\|_{\pazocal{L}^{p}}^p} \leq \sup_{F\in \pazocal{F}^1(\mathbb{C}^d)\setminus \{0\}} \frac{\|F\cdot \chi_{\Omega}\|_{\pazocal{L}^1}}{\|F\|_{\pazocal{L}^{1}}}.
\end{align*}
\end{lemma}

The proof combines central results from \textit{complex interpolation}. 
An introduction to complex interpolation, based on Hadamard's three line theorem, can be found in \cite{Zhu_Operator_theory_in_function_spaces} Chapter 2. We follow the notation in \cite{Zhu_Operator_theory_in_function_spaces} and let $X_{\theta}=[X_0,X_1]_{\theta}$ for $\theta\in[0,1]$ denote the complex interpolation space between two (compatible) Banach spaces $X_0$ and $X_1$. 
To begin with, we present a general result regarding linear operators between interpolation spaces and the associated operator norm (from \cite{Zhu_Operator_theory_in_function_spaces} Theorem 2.4 (c) combined with subsequent remark).

\begin{theorem}
\label{theorem_complex_interpolation_linear_operator_estimate}
Suppose $X_{0}, X_{1}$ and $Y_{0},Y_{1}$ are compatible pairs of Banach spaces, and suppose the mapping
\begin{align*} 
    T: X_{0}+X_{1} \to Y_{0}+Y_{1}
\end{align*}
is bounded, linear such that $T: X_{k}\to Y_{k}$ is bounded with norm $\|T\|_{k} \leq M_{k}$ for $k=0,1$. Then the mapping $T$ satisfies $T: [X_0,X_1]_{\theta}\to [Y_0,Y_1]_{\theta}$ with norm estimate $\|T\|_{\theta}\leq M_{0}^{1-\theta}M_{1}^{\theta}$ for all $\theta\in[0,1]$.
\end{theorem}
In order to relate the above theorem to our context, we need to consider interpolation between weighted $L^p$-spaces and between Fock spaces. The next statements are all found in \cite{Zhu_Analysis_on_Fock_spaces} Chapter 2.4, formulated for spaces over $\mathbb{C}$ but easily generalized to $\mathbb{C}^d$. First, we consider the \textit{Stein-Weiss} interpolation theorem, first published in \cite{Stein_Weiss}, for weighted $L^p$-spaces.
\begin{theorem}
Suppose $w,w_0$ and $w_1$ are positive weight functions on $\mathbb{C}^{d}$. Then for any $1\leq p_0\leq p_1 \leq \infty$ and $\theta\in [0,1]$, we have that 
\begin{align*}
    \left[L^{p_0}(\mathbb{C}^{d},w_0\mathrm{d}A), L^{p_1}(\mathbb{C}^{d}, w_1\mathrm{d}A)\right]_{\theta} = L^{p}(\mathbb{C}^{d},w\mathrm{d}A),
\end{align*}
where
\begin{align*}
    \frac{1}{p} = \frac{1-\theta}{p_{0}}+\frac{\theta}{p_1} \ \ \text{and } \ w^{\frac{1}{p}} = w_0^{\frac{1-\theta}{p_0}}w_1^{\frac{\theta}{p_1}}.
\end{align*}
\end{theorem}
In particular, for the $L^p$-spaces with Gaussian measures, $\pazocal{L}^p(\mathbb{C}^d)$, we obtain:  
\begin{corollary}
\label{corollary_interpolation_Gaussian_weighted_Lp_spaces}
For any $1\leq p_0 \leq p_1 \leq \infty$ and $\theta\in[0,1]$, we have that
\begin{align*}
    \left[\pazocal{L}^{p_0}(\mathbb{C}^d), \pazocal{L}^{p_1}(\mathbb{C}^d) \right]_{\theta} = \pazocal{L}^{p}(\mathbb{C}^d), \ \ \text{where} \ \ \frac{1}{p} = \frac{1-\theta}{p_{0}}+\frac{\theta}{p_1}.
\end{align*}
\end{corollary}
For Fock spaces we have the following interpolation:
\begin{theorem}
\label{theorem_interpolation_Fock_spaces}
For any $1\leq p_0 \leq p_1 \leq \infty$ and $\theta\in[0,1]$, we have that
\begin{align*}
    \left[\pazocal{F}^{p_0}(\mathbb{C}^d), \pazocal{F}^{p_1}(\mathbb{C}^d) \right]_{\theta} = \pazocal{F}^{p}(\mathbb{C}^d), \ \ \text{where} \ \ \frac{1}{p} = \frac{1-\theta}{p_{0}}+\frac{\theta}{p_1}.
\end{align*}
\end{theorem}

With these interpolation results, we are ready to prove Lemma \ref{lemma_complex_interpolation_characteristic_function}.
\begin{proof}
(Lemma \ref{lemma_complex_interpolation_characteristic_function}) 
Note that the statement is equivalent to the linear mapping 
\begin{align*}
    &T_{\Omega}:\pazocal{F}^p(\mathbb{C}^d)\to \pazocal{L}^p(\mathbb{C}^d), \ \ F\mapsto F\cdot \chi_{\Omega}
\intertext{satisfying the operator norm inequality $\|T_{\Omega}\|_{\pazocal{L}^p}^{p}\leq \| T_{\Omega}\|_{\pazocal{L}^1}$ for $p\geq 1$. Therefore, we consider the mapping $T_{\Omega}$ between the spaces} 
    &T_{\Omega}: \pazocal{F}^{1}(\mathbb{C}^d)+\pazocal{F}^{\infty}(\mathbb{C}^d)\to \pazocal{L}^1(\mathbb{C}^d)+\pazocal{L}^{\infty}(\mathbb{C}^d),
\intertext{which is clearly bounded. By Theorem \ref{theorem_complex_interpolation_linear_operator_estimate},} 
    &T_{\Omega}: \left[\pazocal{F}^{1}(\mathbb{C}^d),\pazocal{F}^{\infty}(\mathbb{C}^d)\right]_{\theta}\to \left[\pazocal{L}^1(\mathbb{C}^d), \pazocal{L}^{\infty}(\mathbb{C}^d)\right]_{\theta},
\end{align*}
and the associated operator norm is bounded by $\| T_{\Omega}\|_{\pazocal{L}^1}^{1-\theta}\|T_{\Omega}\|_{\pazocal{L}^{\infty}}^{\theta}$. By Corollary \ref{corollary_interpolation_Gaussian_weighted_Lp_spaces} and Theorem \ref{theorem_interpolation_Fock_spaces}, these interpolation spaces correspond to $\pazocal{F}^{p}(\mathbb{C}^d)$ and $\pazocal{L}^p(\mathbb{C}^d)$, respectively, with $p^{-1}=1-\theta$. In addition, $\|T_{\Omega}\|_{\pazocal{L}^{\infty}} \leq 1$, from which the norm estimate readily follows. 
\end{proof}

\section{Omitted proof: Simple porosity estimate of Cantor sets}
\label{appendix_porosity_estimate_cantor_set}
We shall prove the following simple porosity estimate for the $n$-iterate Cantor set in $1$-dimension.  
\begin{lemma}
The $n$-iterate Cantor set $\pazocal{C}_n(L,M,\mathcal{A})$ with base $M>1$ and alphabet size $|\mathcal{A}|<M$, based in the interval $[0,L]$, is $\nu$-porous on scales $L M^{-n+1}$ to $\infty$, with any $\nu \leq M^{-2}$. 
\begin{proof}
Without loss of generality, we assume $L=1$. Consider an interval $I\subseteq \mathbb{R}$ of size $M^{-m+1}\leq |I| \leq M^{-m+2}$ for some integer $1\leq m\leq n$. Suppose first that the intersection $I\cap \pazocal{C}_m(\dots)$ does \textit{not} form an interval. Then, by the Cantor set construction, there exists an interval $J\subseteq I$ of size $|J|= M^{-m}\geq M^{-2}|I|$ such that $|J\cap \pazocal{C}_{m}(\dots)|=0$. Conversely, suppose the intersection forms an interval, effectively dividing the remainder $I\setminus \pazocal{C}_{m}(\dots)$ into two intervals $J_1, J_2$. Again, by the Cantor set construction, we have the upper bound $|I\cap \pazocal{C}_{m}(\dots)|\leq |\mathcal{A}|M^{-m}\leq (M-1)M^{-m}$, so that $\max\{|J_1|, |J_2|\} \geq \frac{1}{2}(|I|-(M-1)M^{-m})$. Hence, we conclude that there exists an interval $J\subseteq I$ with $|I\cap \pazocal{C}_n(\dots)|=0$ of size
\begin{align*}
    |J| =\left(\frac{|I|-(M-1)M^{-m}}{2|I|}\right)|I|\geq M^{-2}|I|.
\end{align*}
Since $m$ is arbitrary and $\pazocal{C}_n(\dots)\subseteq \pazocal{C}_m(\dots)$, the statement holds for all intervals $I$ with $M^{-n+1}\leq |I|\leq M$. For $|I|\geq M$ the statement becomes trivial as $\pazocal{C}_n(\dots)\subset[0,1]$.
\end{proof}
\end{lemma}

\section{Subaveraging in Fock space}
\label{appendix_subaveraging_proof}
We shall generalize the subaveraging statement made in Lemma 2.32 in \cite{Zhu_Analysis_on_Fock_spaces} for the Fock space over $\mathbb{C}$ to the space over $\mathbb{C}^{d}$, namely: 
\begin{lemma}
\label{lemma_subaveraging_in_Fock_space_appendix}
For any $F\in \pazocal{F}^p(\mathbb{C}^{d})$ and any point $z\in \mathbb{C}^{d}$, we have for all $R>0$ that
\begin{equation}
\begin{aligned}
    |F(z)|^p e^{-\frac{p}{2}\pi |z|^2}\leq \left(\frac{p}{2}\right)^{d}\left(1-e^{-\frac{p}{2}\pi R^2}\sum_{j=0}^{d-1}\frac{\left(\frac{p}{2}\pi R^2\right)^{j}}{j!}\right)^{-1} \int_{B_{R}(z)}|F(\xi)|^p e^{-\frac{p}{2}\pi|\xi|^2}\mathrm{d}A(\xi).
\end{aligned}
\end{equation}
\end{lemma}

In one dimension, the proof is based on the subaveraging property of subharmonic functions in $\mathbb{C}$. More precisely, for $u:\mathbb{C}\to \mathbb{R}\cup \{-\infty\}$ subharmonic, we have that 
\begin{align*}
    u(z) \leq \frac{1}{2\pi}\int_{0}^{2\pi}u(re^{i\theta}+z)\mathrm{d}\theta \ \ \forall \ r>0.  
\end{align*}
In $d$-dimensions, we instead consider plurisubharmonic functions. 
\begin{definition} \textnormal{(Plurisubharmonic)}
Let $X$ be a domain in $\subseteq\mathbb{C}^{d}$. We say that a function $u:X\to \mathbb{R}\cup\{-\infty\}$ is plurisubharmonic if
\begin{enumerate}[label =  \textnormal{(\alph*)},itemsep=0.4ex, before={\everymath{\displaystyle}}]
    \item $u$ is upper semi-continuous, and
    \item for every $z,\xi\in\mathbb{C}^{d}$ the function $\tau \mapsto u(\xi\tau+z)$ is subharmonic in the open subset of $\mathbb{C}$ where it is defined.
\end{enumerate}
\label{definition_plurisubharmonic}
\end{definition}

Many well-known examples of plurisubharmonic functions are based on the holomorphic functions. In particular, if $F$ is entire, then $|F|^{p}$ is a plurisubharmonic function in $\mathbb{C}^{d}$ for every $p>0$. Hence, by point (b) in Def. \ref{definition_plurisubharmonic}, every $F\in \pazocal{F}^{p}(\mathbb{C}^{d})$ satisfies
\begin{align}
    |F(z)|^p \leq \frac{1}{2\pi}\int_{0}^{2\pi}|F|^p(\xi r e^{i\theta}+z)\mathrm{d}\theta \ \ \forall \ \xi \in \mathbb{C}^{d} \ \ \text{and} \ \ r>0.
    \label{entire_functions_subaveraging_circle}
\end{align}
Notice that since left-hand side of \eqref{entire_functions_subaveraging_circle} is independent of $\xi$ (and $r$), we may multiply both sides by a positive function of $\xi$, integrate with respect to $\xi$ and keep the inequality \textit{intact}. 

Introduce the notation 
\begin{align*}
    \partial B_1 := \big\{z\in\mathbb{C}^{d} \ \big| \ |z|=1\big\}
\end{align*}
for the $(2d-1)$-dimensional unit sphere, and let $|\partial B_{1}|$ denote the associated surface area. Further, let $\mathrm{d}S$ denote the surface measure on $\partial B_1$.
To begin with, we show that the point evaluation $|F(z)|^2$ is bounded by an average over the sphere. 
\begin{lemma}
For every entire function $F$, we have that
\begin{align}
    |F(z)|^p \leq \frac{1}{|\partial B_1|}\int_{|\xi|=1}|F|^p(\xi r + z)\mathrm{d}S(\xi) \ \ \forall \ r>0.
    \label{entire_function_subaverage_sphere}
\end{align}
\begin{proof} 
Consider the equivalence relation $\sim$ in $\mathbb{C}^{d}$, where 
\begin{align*}
    \xi_{1} \sim \xi_{2} \iff \exists  \ \ \theta\in [0,2\pi] \ \text{such that} \ \xi_{1} = e^{i\theta}\xi_2.
\end{align*}
This induces the quotient space $\partial B_1/\sim$ of "remaining angles" over the unit sphere, with measure $\mathrm{d}\Phi$ such that $\mathrm{d}S = \mathrm{d}\theta\mathrm{d}\Phi$ and 
\begin{align*}
    |\partial B_1| =\int_{0}^{2\pi}\mathrm{d}\theta\int_{\partial B_1/\sim}\mathrm{d}\Phi= 2\pi \int_{\partial B_1/\sim}\mathrm{d}\Phi.
\end{align*}
With the above decomposition, we may integrate both sides of \eqref{entire_functions_subaveraging_circle} over the space $\partial B_1/\sim$, which leaves the desired result. 
\end{proof}
\end{lemma}

Proceeding, we relate subaverage over the sphere to an subaverage over the ball $B_{R}(0)$. 

\begin{lemma}
\label{lemma_subaverage_ball}
For every entire function $F$, we have for all $R>0$ that
\begin{equation}
\begin{aligned}
    |F(z)|^p\leq \left(\frac{p}{2}\right)^{d}\left(1-e^{-\frac{p}{2}\pi R^2}\sum_{j=0}^{d-1}\frac{\left(\frac{p}{2}\pi R^2\right)^{j}}{j!}\right)^{-1} \int_{B_R(0)}|F|^p(z+\xi)e^{-\frac{p}{2}\pi |\xi|^2}\mathrm{d}A(\xi).
    \label{entire_function_subaverage_ball}
\end{aligned}
\end{equation}
\begin{proof}
For any $\omega \in \mathbb{C}^{d}$, we have the basic decomposition $\omega = r\xi$ for $|\xi| = 1$ and $r = |\omega|$, so that the volume measure can be expressed $\mathrm{d}A(\omega) = r^{2d-1}\mathrm{d}r\mathrm{d}S(\xi)$. Since we consider the Gaussian measure of the Fock space $\pazocal{F}^{p}(\mathbb{C}^{d})$, we integrate both sides of inequality \eqref{entire_function_subaverage_sphere} against $e^{-\frac{p}{2}\pi r^2} r^{2d-1}\mathrm{d}r$. For the right-hand side, we find that
\begin{align*}
    \frac{1}{|\partial B_1|}\int_{0}^{R}\int_{|\xi|=1}|F|^{p}(\xi r + z)e^{-\frac{p}{2}\pi r^2}r^{2d-1} \mathrm{d}S(\xi)\mathrm{d}r = \frac{1}{|\partial B_1|} \int_{B_{R}(0)}|F|^{p}(\omega+z)e^{-\frac{p}{2}\pi |\omega|^2}\mathrm{d}A(\omega).
\end{align*}
While for the left-hand side, using the formula 
\begin{align*}
    \int_{0}^{L}r^{k}e^{-r}\mathrm{d}r = k! \left(1-e^{-L}\sum_{j=0}^{k}\frac{L}{j!}\right) \ \ \text{for } k=0,1,2,\dots,
\end{align*}
we obtain
\begin{align*}
    |F(z)|^p\int_{0}^{R}e^{-\frac{p}{2}\pi r^2} r^{2d-1}\mathrm{d}r = |F(z)|^p \frac{(d-1)!}{2}\left(\frac{2}{p\pi}\right)^d \left(1-e^{-\frac{p}{2}\pi R^2}\sum_{j=0}^{d-1}\frac{\left(\frac{p}{2}\pi R^2\right)^j}{j!}\right).
\end{align*}
Since the surface area of the $(2d-1)$-dimensional unit sphere is given by $|\partial B_1| = 2 \pi^{d}/(d-1)!$, inequality \eqref{entire_function_subaverage_ball} follows. 
\end{proof}
\end{lemma}

We are now ready to prove Lemma \ref{lemma_subaveraging_in_Fock_space_appendix}.

\begin{proof}
(Lemma \ref{lemma_subaveraging_in_Fock_space_appendix})
Consider the integral over the ball centered at $z\in\mathbb{C}^{d}$, 
\begin{equation}
\begin{aligned}
    \int_{B_{R}(z)}|F(\xi)|^{p}e^{-\frac{p}{2}\pi|\xi|^{2}}\mathrm{d}A(\xi)= &\int_{B_{R}(0)}|F|^{p}(z+\xi)e^{-\frac{p}{2}\pi|z+\xi|^{2}}\mathrm{d}A(\xi)\\ = e^{-\frac{p}{2}\pi|z|^2}&\int_{B_{R}(0)}|F(z+\xi)e^{-\pi\xi\cdot\overline{z}}|^{p}e^{-\frac{p}{2}\pi|\xi|^{2}}\mathrm{d}A(\xi). 
    \label{subaverage_Fock_space_intermediate_computation}
\end{aligned}
\end{equation}
Define $G_{z}(\xi) := F(z+\xi)e^{-\pi\xi\cdot\overline{z}}$, which is an entire function with respect to $\xi$. Since $G_{z}(0)=F(z)$, the inequality follows once we apply Lemma \ref{lemma_subaverage_ball} to $|G_{z}(0)|^{p}$, combined with \eqref{subaverage_Fock_space_intermediate_computation}.
\end{proof}
 
\section*{Acknowledgements}
The research of the author was supported by Grant 275113 of the Research Council of
Norway. The author would like to extend thanks to prof. Eugenia Malinnikova for many insightful
discussions that helped shape the ideas in the manuscript. In addition, the author would like to thank prof. Franz Luef for the support and who suggested that fractal uncertainty principles for the Daubechies' operator could be translated to the context of Gabor multipliers. 

\addcontentsline{toc}{section}{References}
\bibliographystyle{abbrv}
\bibliography{A_Fractal_Uncertainty_Principle_for_the_STFT_and_Gabor_multipliers}

\end{document}